\documentclass[reqno]{amsart}
\usepackage{amscd}
\usepackage[dvips]{graphics}

\def\beq{\begin{equation}}
\def\eeq{\end{equation}}
\def\ba{\begin{array}}
\def\ea{\end{array}}

\def\R{\mathbb R}

\newtheorem{thm}{Theorem}[section]
\newtheorem{lm}[thm]{Lemma}
\newtheorem{prop}[thm]{Proposition}

\theoremstyle{definition}
\newtheorem{rem}[thm]{Remark}

\theoremstyle{remark}

\begin{document}
\pagestyle{plain}
\title{Anisotropic Moser-Trudinger inequality involving $L^{n}$ norm}

\author{Changliang Zhou\\
 \small\small {School of Science, East China University of Technology, Nanchang, China }\\
}

\thanks{The authors are supported partially by NSFC of China (No. 11771285).
The corresponding author are  Changliang Zhou, zhzhl800130@sjtu.edu.cn
  }

\begin{abstract} The paper is concerned aboout a sharp form of  Anisotropic Moser-Trudinger inequality which  involves  $L^{n}$ norm. Let
  \begin{equation*}
  \lambda_{1}(\Omega)=\inf_{u\in W_{0}^{1,n}(\Omega),u\not\equiv 0}||F(\nabla u)||_{L^{n}(\Omega)}^{n}/||u||_{L^{n}(\Omega)}^{n}
  \end{equation*}
  be the first eigenvalue associated with $n$-Finsler-Laplacian. using blowing up analysis, we obtain that 
  \begin{equation*}
  \sup_{u\in W_{0}^{1,n}(\Omega),||F(\nabla u)||_{L^{n}(\Omega)}=1}\int_{\Omega}e^{\lambda_{n}(1+\alpha||u||_{L^{n}(\Omega)}^{n})^{\frac{1}{n-1}}|u|^{\frac{n}{n-1}}}dx
  \end{equation*}
  is finite for any $0\leq \alpha<\lambda_{1}(\Omega)$,and the supremum is infinite for any $\alpha\geq \lambda_{1}(\Omega)$, where $\lambda_{n}=n^{\frac{n}{n-1}}\kappa_{n}^{\frac{1}{n-1}}$ ($\kappa_{n}$ is the volume of the unit wulff ball)
   and the function $F$ is positive,convex and homogeneous of degree $1$, and its polar $F^o$ represents a Finsler metric on $\R^n$. Furthermore, the supremum is attained for any $0\leq \alpha<\lambda_{1}(\Omega)$.
\end{abstract}

\maketitle

{\bf Keywords:} Moser-Trudinger inequality, $n$-Finsler-Laplacian, Blow-up analysis\\

{\bf MSC: }35A05, 35J65

 \section{introduction}
 Let $\Omega\subset \R^n $ be a smooth bounded domain. The Sobolev embedding theorem states that $W_{0}^{1,n}(\Omega)$ is embedded in $L^{p}(\Omega)$ for any $p>1$,  or equivalently, using the Dirichlet norm $\|u\|_{W^{1,n}_0(\Omega)}=(\int_{\Omega}|\nabla u|^ndx)^{\frac 1n} $  on $W_{0}^{1,n}(\Omega)$,
  $$\sup_{u\in W_{0}^{1,n}(\Omega),||\nabla u||_{L^{n}(\Omega)}\leq1}\int_{\Omega}|u|^pdx<+\infty.$$
  But it is well known that $W_{0}^{1,n}(\Omega)$ is not embedded in $ L^{\infty}(\Omega)$. Hence, one is led to look for a function $g(s): \R \rightarrow \R^+ $ with maximal growth such that
  $$\sup_{u\in W_{0}^{1,n}(\Omega),||\nabla u||_{L^{n}(\Omega)}\leq1}\int_{\Omega}g(u)dx<+\infty.$$  The Moser-Trudinger inequality states that the maximal growth is of exponential type, which was shown by Pohozhaev \cite{S2}, Trudinger \cite{T} and Moser \cite{M}. This inequality says that
\begin{equation}\label{1-01}
\sup_{u\in W_{0}^{1,n}(\Omega),||\nabla u||_{L^{n}(\Omega)}\leq1}\int_{\Omega}e^{\alpha |u|^{\frac{n}{n-1}}}dx<+\infty
\end{equation}
for any $\alpha\leq\alpha_{n}$, where $\alpha_{n}=n\omega_{n-1}^{\frac{1}{n-1}}$  and $\omega_{n-1}$ is the surface area of the unit ball in $\R^n$. The inequality is optimal, i.e. for any $\alpha>\alpha_{n}$ there exists a sequence of $\{u_{\epsilon}\}$ in $W_{0}^{1,n}(\Omega)$ with  $||\nabla u_\epsilon||_{L^{n}(\Omega)}\leq 1 $ such that
\begin{equation*}
\int_{\Omega}e^{\alpha |u_{\epsilon}|^{\frac{n}{n-1}}}dx\rightarrow \infty \qquad as~~\epsilon\rightarrow 0.
\end{equation*}
On the other hand, for any fixed $u\in W_{0}^{1,n}(\Omega)$, it is also known that
\begin{equation*}
\int_{\Omega}e^{\alpha |u|^{\frac{n}{n-1}}}dx<+\infty
\end{equation*}
for any $\alpha>0$.

Another interesting question about Moser-Trudinger inequalities is whether extremal functions  exist or not. The first result in this direction is due to Carleson and Chang \cite{CC}, who proved that the supremum is attained when $\Omega$ is a unit ball in $\R^{n}$. Then Struwe \cite{S} got the existence of extremals for $\Omega$ close to a ball. Struwe's technique was then used and extended by Flucher \cite{F} to $\Omega$ which is the more general bounded smooth domain in $\R^2$. Later, Lin \cite{L2} generalized the existence result to a bounded smooth domain in  dimension-$n$. Recently  Mancini and  Martinazzi \cite{MM} reproved the Carleson and Chang's result by using a new method based on the Dirichlet energy, also allowing for perturbations of the functional. Even Thizy \cite{T1} also gave examples in which slightly perturbed Moser-Trudinger inequalities do not admit extremals (as conjectured in \cite{MM}), hence shown that the existence of extremal for the Moser-Trudinger inequality should not be taken for granted and holds only under some rigid conditions.
Actually, the inequality (\ref{1-01}) is viewed as a $n$-dimensional analog of the Sobolev inequality, and it plays an important role in  analytic problems and in geometric problems. Now there are many generalizations of the classical Moser-Trudinger inequality (\ref{1-01}), see for instance \cite{AD,L1,LY,IM,CR,Y1,Y2,YZ1,YZ2,WY} and the references therein.


In 2012,  Wang and  Xia \cite{WX} proved the following Moser-Trudinger type inequality
\begin{equation}\label{1-02}
\int_{\Omega}e^{\lambda |u|^{\frac{n}{n-1}}}dx\leq C(n)|\Omega|
\end{equation}
for all $u\in W_{0}^{1,n}(\Omega)$ and $\int_{\Omega}F^{n}(\nabla u)dx\leq 1$. Here $\lambda\leq\lambda_{n}=n^{\frac{n}{n-1}}\kappa_n^{\frac{1}{n-1}}$, $\lambda_{n}$ is optimal in the sense that if $\lambda >\lambda_{n}$ we can find a sequence $\{u_{k}\}$ such that $\int_{\Omega}e^{\lambda |u_k|^{\frac{n}{n-1}}}dx$ diverges. Later, in \cite{ZZ} and \cite {ZZ1} there have  shown that  the supremum is attained when $\Omega$  is bounded domain in  $\R^{n}$.

Adimurthi and Druet \cite{AD}, Y.Y Yang \cite{Y1} have proved that, when $\overline{\lambda}_{1}(\Omega)>0$ be the first eigenvalue of the $n$-Laplacian with Dirichlet boundary condition in $\Omega$, then \\
(1) For any $0\leq \alpha<\overline{\lambda}_{1}(\Omega)$,
\begin{equation*}
\sup_{u\in W_{0}^{1,n}(\Omega),||\nabla u||_{L^{n}(\Omega)}^{n}=1}\int_{\Omega}e^{\alpha_{n}|u|^{\frac{n}{n-1}}(1+\alpha||u||_{L^{n}(\Omega)}^{n})^{\frac{1}{n-1}}}dx<+\infty.
\end{equation*}
(2) For any $\alpha\geq\overline{\lambda}_{1}(\Omega)$,
\begin{equation*}
\sup_{u\in W_{0}^{1,n}(\Omega),||\nabla u||_{L^{n}(\Omega)}^{n}=1}\int_{\Omega}e^{\alpha_{n}|u|^{\frac{n}{n-1}}(1+\alpha||u||_{L^{n}(\Omega)}^{n})^{\frac{1}{n-1}}}dx=+\infty.
\end{equation*}
Furthermore, when $0\leq \alpha<\overline{\lambda}_{1}(\Omega)$, the supremum is also attained.

For simplicity, we introduce the notions
\begin{equation}\label{1-003}
J_{\lambda}^{\alpha}(u)=\int_{\Omega}e^{\lambda(1+\alpha||u||_{L^{n}(\Omega)}^{n})^{\frac{1}{n-1}}|u|^{\frac{n}{n-1}}}dx,
\mathcal{H}=\{u\in W_{0}^{1,n}(\Omega):||F(\nabla u)||_{L^{n}(\Omega)}=1\}.
\end{equation}

Let $\lambda_{1}(\Omega)$ be the first eigenvalue of Finsler-n-Laplacian with Dirichlet boundary condition in $\Omega$. It is defined by
\begin{equation}\label{1-004}
\lambda_{1}(\Omega)=\inf_{u\in W_{0}^{1,n}(\Omega),u\neq 0}\frac{||F(\nabla u)||_{L^{n}(\Omega)}^{n}}{||u||_{L^{n}(\Omega)}^{n}},
\end{equation}
from \cite{BFK}, we can know that $\lambda_{1}(\Omega)>0$ and it is also achieved uniquely by a positive function $\varphi$ satisfying
\begin{equation*}
\left\{
\begin{array}{llc}
-Q_{n}\varphi=\lambda_{1}(\Omega)|\varphi|^{n-2}\varphi,~~~&\text{in}~~\Omega\\
\varphi=0            &\text{on}~~\Omega,
\end{array}
\right.
\end{equation*}
where $-Q_{n}$ is Finsler-Laplacian operator that can be founded in section 2.\\
In this paper we state the following:

\begin{thm}\label{1-04}
Let $\Omega\subset \R^n$ be  a smooth and bounded domain and let $\lambda_{1}(\Omega)>0$ be the first eigenvalue of the Finsler-n-Laplacian with Dirichlet boundary condition in $\Omega$.Then we have
\begin{eqnarray*}
&(1)&~~\text {For any}~~~0\leq \alpha<\lambda_{1}(\Omega),~~\sup_{u\in\mathcal{H}}J_{\lambda_{n}}^{\alpha}(u)<+\infty.\\
&(2)&~~\text{For any}~~~\alpha\geq\lambda_{1}(\Omega),~~\sup_{u\in\mathcal{H}}J_{\lambda_{n}}^{\alpha}(u)=+\infty.
\end{eqnarray*}
\end{thm}

\begin{thm}\label{1-05}
Let $\Omega\subset \R^n$ be  a smooth and bounded domain. For any $0\leq \alpha<\lambda_{1}(\Omega)$, $\sup_{u\in\mathcal{H}}J_{\lambda_{n}}^{\alpha}(u)$ is attained by some $C^{1}$ maximizer. In other words, there exists $u_{\alpha}\in \mathcal{H}\cap C^{1}(\Omega)$ such that $J_{\lambda_{n}}^{\alpha}(u_{\alpha})=\sup_{u\in\mathcal{H}}J_{\lambda_{n}}^{\alpha}(u)$.
\end{thm}
When $F(x)=|x|$,  Adimurthi and Druet \cite{AD}, Y.Y Yang \cite{Y1} have proved the above theorem. But the $F(x)\neq |x|$, it is more different, need much more delicate work. Now we describe the main idea to prove Theorem \ref{1-04} and Theorem \ref{1-05}. The proof of point (2) of Theorem \ref{1-04} is base on test functions computations which are present in Section 2. The point of point (1) of Theorem \ref{1-04} is based on the blow up analysis. The proof of Theorem \ref{1-05} is based on two facts: an upper bound of $J_{\lambda_{n}}^{\alpha}$ on $\mathcal{H}$ can be derived under the assumption that blowing up occur; a sequence of functions $\phi_{\epsilon}\in \mathcal{H}$ can be constructed to show that the above upper bound is not an upper bound. This contradiction implies that no blowing up occur, and then Theorem \ref{1-05} holds. Though the method we carry out blowing up analysis is routine, we will encounter new difficulties when $0<\alpha<\lambda_{1}(\Omega)$

We organize this paper as follows. In Section 2, we gives some notes about anisotropic function $F(x)$ and the properties of the function $F(x)$, moreover, we prove point (2) of Theorem \ref{1-04}. we use blowing up analysis to prove point of (1) of Theorem \ref{1-04} in section 3 and section 4. An upper bound of $J_{\lambda_{n}}^{\alpha}$ is derived ,moreover a sequence of functions is constructed to reach a contradiction in section 5, which completes the proof of Theorem \ref{1-05}. In section 6, we show the asymptotic representation of certain Green function which has been used in Section 5.
\section{notations and preliminaries }
In this section, we will give the notations and preliminaries.

Throughout this paper, let  $F :\R^n\mapsto \R$  be a nonnegative convex function of class $C^{2}(\R^n\backslash\{0\})$  which is even and positively homogenous of degree  $1$,  so that
$$F(t\xi)=|t|F(\xi)\qquad \text{for any}\qquad t\in \R,~~~~\xi\in \R^n.$$
A typical example is  $F(\xi)=(\sum_{i}|\xi_i|^{q})^{\frac{1}{q}}$  for  $q\in [1,\infty)$.  We further assume that
$$F(\xi)>0\qquad \text{for any}\qquad\xi\neq 0.$$

Thanks to homogeneity of  $F$,  there exist two constants  $0<a\leq b<\infty$  such that
\begin{equation*}
a|\xi|\leq F(\xi)\leq b|\xi|.
\end{equation*}
 Usually, we shall assume that the $Hess(F^{2})$ is positive definite in  $\R^n\backslash\{0\}$. Then by Xie and Gong \cite{XG}, $Hess(F^{n})$ is also positive definite in  $\R^n\backslash\{0\}$.
If we consider the minimization problem
$$
\min_{u\in W^{1,n}_0(\Omega)}\int_{\Omega}F^n(\nabla u)dx,
$$
its Euler equation contains an  operator of the form
$$
Q_{n}u:=\sum_{i=1}^{i=n}\frac{\partial}{\partial x_i}(F^{n-1}(\nabla u)F_{\xi_i}(\nabla u)).
$$ Note that these operators are not linear unless $F$ is the Euclidean norm in dimension two. We call this nonlinear operator as n-Finsler-Laplacian. This operator $Q_{n}$ was studied by many mathematicians, see \cite{WX,FK,WX,AVP,BFK,XG} and the references therein.

Consider the map
$$\phi:S^{n-1}\rightarrow \R^n, ~~~~~ \phi(\xi)=F_{\xi}(\xi).$$
Its image  $\phi(S^{n-1})$  is a smooth, convex hypersurface in  $\R^n$, which is called Wulff shape of  $F$.
Let  $F^{o}$  be the support function of  $K:=\{x\in \R^n:F(x)\leq 1\}$,  which is defined by
$$F^{o}(x):=\sup_{\xi\in K}\langle x,\xi\rangle.$$
It is easy to verify that  $F^{o}:\R^n\mapsto [0,+\infty)$  is also a convex, homogeneous function of class of $C^{2}(\R^n\backslash\{0\})$.  Actually  $F^{o}$  is dual to   $F$  in the sense that
$$F^{o}(x)=\sup_{\xi\neq 0}\frac{\langle x,\xi\rangle}{F(\xi)},\qquad F(x)=\sup_{\xi\neq 0}\frac{\langle x,\xi\rangle}{F^{o}(\xi)}.$$
One can see easily that  $\phi(S^{n-1})=\{x\in \R^n~|F^{o}(x)=1\}$. We denote $\mathcal{W}_{F}:=\{x\in \R^n|F^{o}(x)\leq 1\}$  and  $\kappa_{n}:=|\mathcal{W}_{F}|$,  the Lebesgue measure of  $ \mathcal{W}_{F}$.  We also use the notion  $\mathcal{W}_{r}(0):=\{x\in \R^n|F^{o}(x)\leq r\}$.  We call $\mathcal{W}_{r}(0)$  a Wulff ball of radius  $r$  with center at  $0$.
For later use, we give some simple properties of the function  $F$,  which follows directly from the assumption on  $F$, also see \cite{WX1,FK,BP}.
\begin{lm}\label{2-01}
We have
\begin{enumerate}
\item[(i)]$|F(x)-F(y)|\leq F(x+y)\leq F(x)+F(y)$;
\item[(ii)]$ \frac{1}{C}\leq|\nabla F(x)|\leq C $, and $ \frac{1}{C}\leq|\nabla F^{o}(x)|\leq C $  for some $C>0$ and any $x\neq 0$;
\item[(iii)]$\langle x,\nabla F(x)\rangle=F(x),\langle x,\nabla F^{o}(x)\rangle=F^{o}(x)$ for any $x\neq 0$;
\item[(iv)]$F(\nabla F^{o}(x))=1$, $F^{o}(\nabla F(x))=1$ for any $x\neq 0$;
\item[(v)]$F^{o}(x) F_{\xi}(\nabla F^{o}(x))=x $ for any $ x\neq 0$;
\item[(vi)]$F_\xi(t\xi)= \text{sgn}(t)F_{\xi}(\xi)$ for any $\xi\neq 0$ and $t\neq 0$.
\end{enumerate}
\end{lm}
Next we describe the isoperimetric inequality and co-area formula with respect to $F$.  For a domain  $\Omega\subset \R^n$,  a subset  $E\subset \Omega$  and a function of bounded variation  $u\in BV(\Omega)$, we define the anisotropic bounded variation of $u$ with respect to $F$ is
$$\int_{\Omega}|\nabla u|_{ F}=\sup\{\int_{\Omega}u \text{ div}\sigma dx,  \sigma\in C_{0}^{1}(\Omega; \R^n), F^{o}(\sigma)\leq 1\}.$$
We set anisotropic perimeter of  $E$  with respect to   $F$ is
$$P_{F}(E):=\int_{\Omega}|\nabla \mathcal{X}_{E}|_{F},$$
where  $\mathcal{X}_{E}$  is the characteristic function of the set  $E$.  It is well known (also see \cite {FM}) that the co-area formula
\begin{equation}\label{2-02}
\int_{\Omega}|\nabla u|_{ F}=\int_{0}^{\infty}P_{F}(|u|>t)dt
\end{equation}
and the isoperimetric inequality
\begin{equation}\label{2-03}
P_{F}(E)\geq n\kappa_{n}^{\frac{1}{n}}|E|^{1-\frac{1}{n}}
\end{equation}
hold. Moreover, the equality in (\ref{2-03}) holds if and only if  $E$  is a Wulff ball.

In the sequel, we will use the convex symmetrization with respect to $F$. The convex symmetrization generalizes the Schwarz symmetrization (see \cite{T3}). It was defined in \cite{AVP} and will be an essential tool for establishing  the Lions type concentration-compactness alternative under the anisotropic Dirichlet norm.
Let us consider a measureed function  $u$  on  $\Omega\subset \R^n$.  The one dimensional decreasing rearrangement of $u$  is
$$u^{*}=\sup\{s\geq 0: |\{x\in\Omega:|u(x)|>s\}|>t\},\qquad \text{for} \quad  t\in \R.$$
The convex symmetrization of  $u$  with respect to  $F$  is defined as
$$u^{\star}(x)=u^{*}(\kappa_{n} F^{o}(x)^{n}),\qquad\text{ for } x\in \Omega^{*}.$$
Here  $\kappa_{n} F^{o}(x)^{n}$  is just Lebesgue measure of a homothetic Wulff ball with radius $F^{0}(x)$ and $\Omega^{*}$  is the homothetic Wulff ball centered at the origin having the same measure as $\Omega$. Throughout this paper, we assume that $\Omega$ is bounded smooth domain in $\mathbb{R}^{n}$ with $n\geq 2$.

The following properties about Finsler-Laplacian can be found in \cite{ZZ1}:
\begin{lm}\label{2-04}
Assume that $u\in W_{0}^{1,n}(\Omega)$ is a solution to the equation
\begin{equation}\label{2-05}
-Q_{n}(u)=f.
\end{equation}
If $f\in L^{q}(\Omega)$ for some $q>1$, then $||u||_{L^{\infty}(\Omega)}\leq C||f||_{L^{q}(\Omega)}^{\frac{1}{n-1}}$, where $C$ is only depends  on $a, b, n, \Omega,q $.
\end{lm}
Set
\begin{equation*}
P=
\left\{
\begin{array}{lcc}
(1-\int_{\Omega}F^{n}(\nabla u)dx)^{-\frac{1}{n-1}}\qquad&\int_{\Omega}F^{n}(\nabla u)dx<1,\\
\infty\qquad&\int_{\Omega}F^{n}(\nabla u)dx=1.
\end{array}
\right.
\end{equation*}
\begin{lm}\label{3-002}
Let $u\in W_{0}^{1,n}(\Omega)$, $u\neq 0$. Assume that  $\{u_{k}\}$ is  a sequence of functions $ W_{0}^{1,n}(\Omega)$ such that $\int_{\Omega}F^{n}(\nabla u_{k})dx\leq 1$, and $$ u_{k}\rightharpoonup u\quad \text{weakly in}\quad W_{0}^{1,n}(\Omega),\quad  u_{k}\rightarrow u\quad  \text { a.e. in } \Omega.$$
Then for every $p<P$, there exists a constant $C=C(n,p)$ such that for each $k$
\begin{equation}\label{3-005}
\int_{\Omega}exp(n^{\frac{n}{n-1}}\kappa_{n}^{\frac{1}{n-1}}p|u_{k}|^{\frac{n}{n-1}})dx\leq C.
\end{equation}
Moreover, this conclusion fails if $p\geq P$.
\end{lm}

{\bf{Test functions computations}}\\
In this subsection we will prove the conclusion (2) of Theorem \ref{1-04}. we will build explicit test functions to show the unboundedness of Moser-Trudinger function under large parameter.

Since the Moser-Trudinger inequality is invariant under translation, we may assume that $0\in \Omega$ and $\mathcal{W}_{1}\subset \Omega$. We now fix some $x_{\delta}\in \mathcal{W}_{1} $ such that $F^{o}(x_{\delta})=\delta$. Choosing $t_{\epsilon}$ such that
$t_{\epsilon}^{n}\log\frac{1}{\epsilon}\rightarrow \infty$ and $t_{\epsilon}^{n+1}\log\frac{1}{\epsilon}\rightarrow 0$.Set
\begin{equation*}
\varphi_{\epsilon}(x)=
\left\{
\begin{array}{lcc}
(\frac{n}{\lambda_{n}}\log\frac{1}{\epsilon})^{\frac{n}{n-1}}, &F^{0}(x)\leq \epsilon,\\
\frac{(\frac{n}{\lambda_{n}}\log\frac{1}{\epsilon})^{\frac{n}{n-1}}(\log\delta-\log F^{0}(x))-t_{\epsilon}\varphi(x_{\delta})(\log \epsilon-\log F^{0}(x))}{\log\delta-\log\epsilon},&\epsilon<F^{0}(x)\leq \delta,\\
t_{\epsilon}[\varphi(x_{\delta})+\theta(x)(\varphi(x)-\varphi(x_{\delta})),&F^{0}(x)>\delta.
\end{array}
\right.
\end{equation*}
In above definition of $\varphi_{\epsilon}(x)$, $\varphi$ is the eigenvalue function. $\theta(x)\in C^{2}(\Omega)$ is a cut-off function satisty $|\nabla \theta(x)|\leq \frac{C}{\delta}$ and
\begin{equation}\label{1-007}
\theta(x)=
\left\{
\begin{array}{lcc}
0, &F^{0}(x)\leq \delta\\
\theta\in (0,1), &\delta<F^{0}(x)<2\delta\\
1,                &F^{0}(x)\geq 2\delta.
\end{array}
\right.
\end{equation}
Let $\delta=\frac{1}{t_{\epsilon}^{n}}\frac{1}{\log\frac{1}{\epsilon}}$, it is easy to see $\epsilon<\delta$ if $\epsilon$ is small enough. We obtain that
\begin{eqnarray*}
\int_{\epsilon<F^{0}(x)\leq \delta} F^{n}(\nabla\varphi_{\epsilon}(x))dx&=&\int_{\epsilon<F^{0}(x)\leq \delta}
\frac{|-(\frac{n}{\lambda_{n}}\log\frac{1}{\epsilon})^{\frac{n}{n-1}}+t_{\epsilon}\varphi(x_{\delta})|^{n}}{|F^{0}(x)|^{n}(\log\delta-\log\epsilon)^{n}}dx\\
&=&\frac{n\kappa_{n}|-(\frac{n}{\lambda_{n}}\log\frac{1}{\epsilon})^{\frac{n}{n-1}}+t_{\epsilon}\varphi(x_{\delta})|^{n}}{(\log\delta-\log\epsilon)^{n-1}}\\
&=&1-n^{\frac{n+1}{n}}\kappa_{n}^{\frac{1}{n}}(\log\frac{1}{\epsilon})^{-\frac{n-1}{n}}t_{\epsilon}\varphi(x_{\delta})(1+o_{\epsilon}(1)),
\end{eqnarray*}
where $o_{\epsilon}(1)\rightarrow 0 $ as $\epsilon\rightarrow 0$. We also have
\begin{eqnarray*}
\int_{\delta\leq F^{0}(x)\leq 2\delta}F^{n}(\nabla \varphi_{\epsilon}(x))dx&=&t_{\epsilon}^{n}\int_{\delta\leq F^{0}(x)\leq 2\delta}|\theta(x)\nabla\varphi_{\epsilon}+\nabla \theta(x)(\varphi(x)+\varphi(x_{\delta}))|^{n}dx\\
&=&t_{\epsilon}^{n}O(\delta^{n})
\end{eqnarray*}
and
\begin{eqnarray*}
\int_{F^{0}(x)\geq 2\delta}F^{n}(\nabla \varphi_{\epsilon}(x))dx&=&t_{\epsilon}^{n}\int_{F^{0}(x)\geq 2\delta}F^{n}(\nabla \varphi(x))dx\\
&=&t_{\epsilon}^{n}(1+O(\delta^{n})).
\end{eqnarray*}
Summing the above integral estimates for $F^{n}(\nabla \varphi_{\epsilon})$ up, we have
\begin{equation*}
\int_{\Omega}F^{n}(\nabla \varphi_{\epsilon}(x))dx=1-n^{\frac{n+1}{n}}\kappa_{n}^{\frac{1}{n}}(\log\frac{1}{\epsilon})^{-\frac{n-1}{n}}t_{\epsilon}\varphi(x_{\delta})
(1+o_{\epsilon}(1))+t_{\epsilon}^{n}(1+O(\delta^{n})).
\end{equation*}
Then
\begin{equation*}
||F(\nabla \varphi_{\epsilon}(x))||_{L^{n}(\Omega)}^{-\frac{n}{n-1}}=
1+\frac{n^{\frac{n+1}{n}}}{n-1}\kappa_{n}^{\frac{1}{n}}(\log\frac{1}{\epsilon})^{-\frac{n-1}{n}}t_{\epsilon}\varphi(x_{\delta})(1+o_{\epsilon}(1))
-\frac{1}{n-1}t_{\epsilon}^{n}(1+O(\delta^{n})).
\end{equation*}
Set $v_{\epsilon}(x)=\frac{\varphi_{\epsilon}(x)}{||F(\nabla \varphi_{\epsilon}(x))||_{L^{n}(\Omega)}}$, then $||F(\nabla v_{\epsilon}(x))||_{L^{n}(\Omega)}=1$. Furthermore,
\begin{eqnarray*}
\lambda_{1}(\Omega)||v_{\epsilon}(x)||_{L^{n}(\Omega)}^{n}&\geq& \frac{\lambda_{1}(\Omega)t_{\epsilon}^{n}}{||F(\nabla \varphi_{\epsilon}(x))||_{L^{n}(\Omega)}^{n}}\int_{F^{o}(x)\geq 2\delta}|\varphi(x)|^{n}dx\\
&\geq& \lambda_{1}(\Omega)t_{\epsilon}^{n}[||\varphi(x)||_{L^{n}(\Omega)}^{n}+O(\delta^{n})][1+n^{\frac{n+1}{n}}\kappa_{n}^{\frac{1}{n}}(\log\frac{1}{\epsilon})^{-\frac{n-1}{n}}t_{\epsilon}\varphi(x_{\delta})(1+o_{\epsilon}(1))\\
& &-t_{\epsilon}^{n}(1+O(\delta^{n}))]\\
&=&t_{\epsilon}^{n}(\lambda_{1}(\Omega)||\varphi(x)||_{L^{n}(\Omega)}^{n}+O(\delta^{n}))(1+O(t_{\epsilon}^{n}))\\
&=&t_{\epsilon}^{n}(1+O(t_{\epsilon}^{n})+O(\delta^{n})),
\end{eqnarray*}
where we have used  $\lambda_{1}(\Omega)||\varphi(x)||_{L^{n}(\Omega)}^{n}=1$.

Next we establish the integral estimates on the domain of $\{x\in \Omega:F^{o}(x)<\epsilon\}$. We have
\begin{eqnarray*}
\lambda_{n}(1+\lambda_{1}(\Omega)||v_{\epsilon}||_{L^{n}(\Omega)}^{n})^{\frac{1}{n-1}}|v_{\epsilon}|^{\frac{n}{n-1}}&\geq&
n\log\frac{1}{\epsilon}(1+\lambda_{1}(\Omega)||v_{\epsilon}||_{L^{n}(\Omega)}^{n})^{\frac{1}{n-1}}||F(\nabla \varphi_{\epsilon})||_{L^{n}(\Omega)}^{\frac{-n}{n-1}}\\
&=&n\log\frac{1}{\epsilon}(1+t_{\epsilon}^{n}(1+O(t_{\epsilon}^{n})+O(\delta^{n})))^{\frac{1}{n-1}}\\
& &\cdot(1+\frac{n^{\frac{n+1}{n}}}{n-1}\kappa_{n}^{\frac{1}{n}}(\log\frac{1}{\epsilon})^{-\frac{n-1}{n}}t_{\epsilon}\varphi(x_{\delta})(1+o_{\epsilon}(1))\\
& &-\frac{1}{n-1}t_{\epsilon}^{n}(1+O(\delta^{n})))\\
&=&n\log\frac{1}{\epsilon}+\frac{n^{\frac{2n+1}{n}}}{n-1}\kappa_{n}^{\frac{1}{n}}(\log\frac{1}{\epsilon})^{\frac{1}{n}}t_{\epsilon}\varphi(x_{\delta})(1+o_{\epsilon}(1))\\
& &-\frac{n}{n-1}\log\frac{1}{\epsilon}t_{\epsilon}^{n}(1+O(\delta^{n}))+\frac{n}{n-1}\log\frac{1}{\epsilon}t_{\epsilon}^{n}(1+O(t_{\epsilon}^{n}))\\
& &+O(\delta^{n})+o_{\epsilon}(1)\\
&=&n\log\frac{1}{\epsilon}+\frac{n^{\frac{2n+1}{n}}}{n-1}\kappa_{n}^{\frac{1}{n}}(\log\frac{1}{\epsilon})^{\frac{1}{n}}t_{\epsilon}\varphi(0)(1+o_{\epsilon}(1)),
\end{eqnarray*}
where the fact that $\varphi(x_{\delta})=\varphi(0)+o_{\epsilon}(1)$ is applied. Note that $\log\frac{1}{\epsilon}t_{\epsilon}^{n}O(\delta^{n})=o_{\epsilon}(1)$.

Considering the above estimates, we deduce that
\begin{eqnarray*}
\int_{\Omega}exp(\lambda_{n}(1+\lambda_{1}(\Omega)||v_{\epsilon}||_{L^{n}(\Omega)}^{n})^{\frac{1}{n-1}}|v_{\epsilon}|^{\frac{n}{n-1}})dx&\geq& Cexp[\frac{n^{\frac{2n+1}{n}}}{n-1}\kappa_{n}^{\frac{1}{n}}(\log\frac{1}{\epsilon})^{\frac{1}{n}}t_{\epsilon}\varphi(0)(1+o_{\epsilon}(1))]\\
&\rightarrow&+\infty
\end{eqnarray*}
as $\epsilon\rightarrow 0$, since $\varphi(0)>0$ and $(\log\frac{1}{\epsilon})^{\frac{1}{n}}t_{\epsilon}\rightarrow +\infty$. Here $C$ is a positive constant independent of $\epsilon$. The conclusion (2) in Theorem \ref{1-04} holds.

\section{maximizers of the subcritical case }
In this section, we will show the existence of the maximizers for Moser-Trudinger functional in the subcritical case. We begin with the following existence  of the maximizers of the subcritical Moser-Trudinger function.

\begin{prop}\label{3-02}
For any small $\epsilon$ and $0\leq \alpha<\lambda_{1}$, there exists some $u_{\epsilon}\in C^{1}(\overline \Omega)\cap \mathcal{H}$  satisfying
\begin{equation*}
J_{\lambda_{n}-\epsilon}^{\alpha}(u_{\epsilon})=\sup_{u\in\mathcal{H}}J_{\lambda_{n}-\epsilon}^{\alpha}(u).
\end{equation*}

\end{prop}
\begin{proof}
For any fixed $\epsilon $, let $\{u_{\epsilon,j}\}$ be a  sequence such that
\begin{equation*}
\lim _{j\rightarrow+\infty}J_{\lambda_{n}-\epsilon}^{\alpha}(u_{\epsilon,j})=\sup_{u\in\mathcal{H}}J_{\lambda_{n}-\epsilon}^{\alpha}(u).
\end{equation*}
Since $u_{\epsilon,j}$ is bounded in $W_{0}^{1,n}(\Omega)$, there exists a subsequence of $u_{\epsilon,j}$ (We do not distinguish subsequence and sequence in the paper) such that
\begin{eqnarray}
 u_{\epsilon,j}&\rightharpoonup& u_{\epsilon}  \text{ weakly in }   W_{0}^{1,n}(\Omega),\nonumber\\
 u_{\epsilon,j}&\rightarrow& u_{\epsilon}   \text{ strongly in }  L^{p}(\Omega),  \text{ for any }   p\geq 1,\nonumber\\
 u_{\epsilon,j}&\rightarrow& u_{\epsilon}   \text{ a.e. }  \Omega\nonumber
\end{eqnarray}
as $j\rightarrow +\infty$. Hence
\begin{equation*}
g_{j}:=exp[(\lambda_{n}-\epsilon)(1+\alpha||u_{\epsilon,j}||_{L^{n}(\Omega)}^{n})^{\frac{1}{n-1}}|u_{\epsilon,j}|^{\frac{n}{n-1}}]\rightarrow
g_{j}:=exp[(\lambda_{n}-\epsilon)(1+\alpha||u_{\epsilon}||_{L^{n}(\Omega)}^{n})^{\frac{1}{n-1}}|u_{\epsilon}|^{\frac{n}{n-1}}]
\end{equation*}
a.e. in $\Omega$.
We claim that $u_{\epsilon}\not\equiv0$, suppose not,$1+\alpha||u_{\epsilon,j}||_{L^{n}(\Omega)}^{n}\rightarrow 1$, from which one can see that $g_{j}$ is bounded in $L^{p}(\Omega)$ for some $p>1$ and $g_{j}\rightarrow 1$ in $L^{1}(\Omega)$. Hence $\sup_{u\in\mathcal{H}}J_{\lambda_{n}-\epsilon}^{\alpha}(u)=|\Omega|$, which is impossible. Therefore $u_{\epsilon}\not\equiv 0$.
Thanks to Lemma \ref{3-002}, for any $q<1/(1-||F(\nabla u_{\epsilon})||_{L^{n}(\Omega)}^{n})^{\frac{1}{n-1}}$, we have
\begin{equation*}
\limsup_{j\rightarrow+\infty}\int_{\Omega}exp[\lambda_{n}q|u_{\epsilon,j}|^{\frac{n}{n-1}}]dx<+\infty.
\end{equation*}
Due to (\ref{1-004}), we get
\begin{equation*}
1+\alpha||u_{\epsilon}||_{L^{n}(\Omega)}^{n}<\frac{1}{1-||F(\nabla u_{\epsilon})||_{L^{n}(\Omega)}^{n}}
\end{equation*}
for $0\leq \alpha<\lambda_{1}$. Thus, $g_{j}$ is bounded in $L^{s}(\Omega)$ for some $s>1$. Since $g_{j}\rightarrow g_{\epsilon}$ a.e. in $\Omega$, we infer that $g_{j}\rightarrow g_{\epsilon}$ strongly in $L^{1}(\Omega)$ as $j\rightarrow +\infty$, Therefore, the extremal function is attained for the case of $\lambda_{n}-\epsilon$ and $||\nabla u_{\epsilon}||_{L^{n}(\Omega)}^{n}=1$.
Clearly we can choose  $u_{\epsilon} \geq 0$. It is not difficult to check that the Euler-Lagrange equation of $u_{\epsilon}$ is
\begin{equation}\label{3-008}
\left\{
\begin{array}{llc}
-Q_{n}u_{\epsilon}=\beta_{\epsilon}\lambda_{\epsilon}^{-1}u_{\epsilon}^{\frac{1}{n-1}}
e^{\alpha_{\epsilon}u_{\epsilon}^{\frac{n}{n-1}}}+\gamma_{\varepsilon}u_{\epsilon}^{n-1},\\
u_{\epsilon}\in W_{0}^{1,n}(\Omega),~~||F(\nabla u_{\epsilon})||_{L^{n}(\Omega)}=1,~~ u_{\epsilon} \geq 0,\\
\alpha_{\epsilon}=(\lambda_{n}-\epsilon)(1+\alpha||u_{\epsilon}||_{L^{n}(\Omega)}^{n})^{\frac{1}{n-1}},\\
\beta_{\epsilon}=(1+\alpha||u_{\epsilon}||_{L^{n}(\Omega)}^{n})/(1+2\alpha||u_{\epsilon}||_{L^{n}(\Omega)}^{n}).\\
\gamma_{\epsilon}=\alpha/(1+2\alpha||u_{\epsilon}||_{L^{n}(\Omega)}^{n}),\\
\lambda_{\epsilon}=\int_{\Omega}u_{\epsilon}^{\frac{n}{n-1}}e^{\alpha_{\epsilon}u_{\epsilon}^{\frac{n}{n-1}}}dx.
\end{array}
\right.
\end{equation}
Since $\beta_{\epsilon}\lambda_{\epsilon}^{-1}u_{\epsilon}^{\frac{1}{n-1}}
e^{\alpha_{\epsilon}u_{\epsilon}^{\frac{n}{n-1}}}$ and $\gamma_{\varepsilon}u_{\epsilon}^{n-1}$ are bounded in $L^{s}(\Omega)$ for some $s>1$, then by Lemma \ref{2-04}, we have  $u_{\epsilon}\in L^{\infty}(\Omega)$. It implies
$\beta_{\epsilon}\lambda_{\epsilon}^{-1}u_{\epsilon}^{\frac{1}{n-1}}
e^{\alpha_{\epsilon}u_{\epsilon}^{\frac{n}{n-1}}}+\gamma_{\varepsilon}u_{\epsilon}^{n-1}\in L^{\infty}(\Omega)$. Then  by Theorem 1 in \cite{L3}, we easily get $u_{\epsilon}\in C^{1,\alpha}(\overline{\Omega})$ for some $\alpha\in (0,1)$, which implies that $u_{\epsilon}\in C^{1}(\overline{\Omega})$

\end{proof}

The following observation is important.

\begin{lm}\label{3-10}
For any $\alpha\in [0,\lambda_{1}(\Omega))$, we have
\begin{equation*}
\lim_{\epsilon\rightarrow 0}J_{\lambda_{n}-\epsilon}^{\alpha}(u_{\epsilon})=\sup_{u\in\mathcal{H}}J_{\lambda_{n}-\epsilon}^{\alpha}(u).
\end{equation*}
\end{lm}
\begin{proof}
Obviously,
\begin{equation*}
\lim_{\epsilon\rightarrow 0}J_{\lambda_{n}-\epsilon}^{\alpha}(u_{\epsilon})\leq\sup_{u\in\mathcal{H}}J_{\lambda_{n}-\epsilon}^{\alpha}(u).
\end{equation*}
On the other hand, for any  $ u\in W_{0}^{1,n}(\Omega)$ with $||F(\nabla u)||_{L^{n}(\Omega)}\leq 1$, from  Fatou's Lemma and Proposition \ref{3-02} we have
\begin{eqnarray*}
\int_{\Omega}e^{\lambda_{n}|u|^{\frac{n}{n-1}}(1+\alpha||u||_{L^{n}(\Omega)}^{n})^{\frac{1}{n-1}}}dx\leq\liminf_{\epsilon\rightarrow 0}\int_{\Omega}e^{(\lambda_{n}-\epsilon)|u|^{\frac{n}{n-1}}(1+\alpha||u||_{L^{n}(\Omega)}^{n})^{\frac{1}{n-1}}}dx\\
\leq \liminf_{\epsilon\rightarrow 0}\int_{\Omega}e^{(\lambda_{n}-\epsilon)|u_{\epsilon}|^{\frac{n}{n-1}}(1+\alpha||u_{\epsilon}||_{L^{n}(\Omega)}^{n})^{\frac{1}{n-1}}}dx,
\end{eqnarray*}
which implies
\begin{equation*}
\lim_{\epsilon\rightarrow 0}J_{\lambda_{n}-\epsilon}^{\alpha}(u_{\epsilon})\geq\sup_{u\in\mathcal{H}}J_{\lambda_{n}-\epsilon}^{\alpha}(u).
\end{equation*}
Hence the result holds.
\end{proof}

\section{blow-up analysis }

In this section, we consider the convergence of the maximizing sequence in section 3. There are two cases. The one case is that  $M_{\epsilon}= max_{\overline{\Omega}}u_{\epsilon}=u_{\epsilon}(x_{\epsilon})$ is bounded. In this case, it is clear that $u_{\epsilon}$ is bounded in $W_{0}^{1,n}(\Omega)$. Then we can assume without loss of generality
\begin{eqnarray}
&u_{\epsilon}\rightharpoonup u_{0}&\qquad\text{ weakly in }  W_{0}^{1,n}(\Omega),\nonumber \\
&u_{\epsilon}\rightarrow u_{0} &\qquad\text{ strongly in }  L^{q}(\Omega),  \forall q\geq 1,\nonumber\\
&u_{\epsilon}\rightarrow u_{0} &\qquad\text{ a.e. in }  \Omega.\nonumber
\end{eqnarray}
Since, for any   $u\in W_{0}^{1,n}(\Omega)$  with  $\int_{\Omega}F^{n}(\nabla u)dx\leq 1$,
by the Lebesgue dominated convergence theorem  we have
\begin{eqnarray}
 \int_{\Omega}e^{\lambda_{n}|u|^{\frac{n}{n-1}}(1+\alpha||u||_{L^{n}(\Omega)}^{n})^{\frac{1}{n-1}}}dx&=&\lim_{\epsilon\rightarrow 0}\int_{\Omega}e^{(\lambda_{n}-\epsilon)|u|^{\frac{n}{n-1}}(1+\alpha||u||_{L^{n}(\Omega)}^{n})^{\frac{1}{n-1}}}dx\nonumber\\
&\leq& \lim_{\epsilon\rightarrow 0}\int_{\Omega}  \int_{\Omega}e^{(\lambda_{n}-\epsilon)|u_{\epsilon}|^{\frac{n}{n-1}}(1+\alpha||u_{\epsilon}||_{L^{n}(\Omega)}^{n})^{\frac{1}{n-1}}}dx\nonumber\\
&=&\int_{\Omega}e^{\lambda_{n}|u_{0}|^{\frac{n}{n-1}}(1+\alpha||u_{0}||_{L^{n}(\Omega)}^{n})^{\frac{1}{n-1}}}dx,\nonumber
\end{eqnarray}
Hence  $u_{0}$  is the desired maximizer.

The other case is that $M_{\epsilon}= u_{\epsilon}(x_{\epsilon})\rightarrow +\infty$ and $x_{\epsilon}\rightarrow x_{0}$   as  $\epsilon\rightarrow 0$. In this case, the maximizing  sequence $u_\epsilon$  blows up as $\epsilon\rightarrow 0 $, where $x_{0}$ is called the blow-up point. In the sequel, we will analysis the  blow-up behaviors of  $u_{\epsilon}$.

 First, by an inequality $e^{t}\leq 1+te^{t}$, we have
\begin{equation}
|\Omega|<\int_{\Omega}  e^{\alpha_{\epsilon}|u_{\epsilon}|^{\frac{n}{n-1}}}dx\leq |\Omega|+\alpha_{\epsilon} \lambda_{\epsilon}.\nonumber
\end{equation}
This leads to $\liminf_{\epsilon \rightarrow 0}\lambda_{\epsilon}>0$.

\
\

\noindent\textbf{Case 1.} $x_{0}$ lies in the interior of $\Omega$.
\begin{lm}\label{4-000}
There holds $u_{0}=0$ and $F^{n}(\nabla u_{\epsilon})dx\rightharpoonup \delta_{x_{0}}$ in the sense of measure as $\epsilon\rightarrow 0$, where $\delta_{x_{0}}$ is the dirac measure at $x_{0}$.
\end{lm}
\begin{proof}
Suppose $u_{0}\not\equiv0$, then for any $\alpha\in [0,\lambda_{1}(\Omega))$, we have
\begin{equation*}
1+\alpha||u_{\epsilon}||_{L^{n}(\Omega)}^{n}\rightarrow 1+\alpha||u_{0}||_{L^{n}(\Omega)}^{n}\leq 1+||F(\nabla u_{0})||_{L^{n}(\Omega)}^{n}<\frac{1}{1-||F(\nabla u_{0})||_{L^{n}(\Omega)}^{n}}.
\end{equation*}
Hence $e^{\alpha_{\epsilon}|u_{\epsilon}|^{\frac{n}{n-1}}}$ is bounded in $L^{s}(\Omega)$ for some $s>1$ provided $\epsilon$
is sufficiently small.
Lemma \ref{2-04} implies that $u_{\epsilon}$
is uniformly bounded in ${\Omega}$. It contradicts the assumption that $M_{\epsilon}\rightarrow +\infty$.

Assume that $F^{n}(\nabla u_{\epsilon})dx\rightharpoonup \mu$ in the sense of measure as $\epsilon\rightarrow 0$, if $\mu\neq \delta_{x_{0}}$, we claims that there exists a cut-off function $\phi(x)\in C_{0}^{1}(\Omega)$, which is supported in $\mathcal{W}_{r}(x_{0})$ for some $r>0$ with $0<\phi(x)<1$ in $\mathcal{W}_{r}(x_{0})\backslash\mathcal{W}_{\frac{r}{2}}(x_{0}) $ and $\phi(x)=1$ in $\mathcal{W}_{\frac{r}{2}}(x_{0})$ satisfying
\begin{equation*}
\int_{\mathcal{W}_{r}(x_{0})} \phi F^{n}(\nabla u_{\epsilon})dx\leq 1-\eta
\end{equation*}
for some $\eta>0$ and small enough $\epsilon$. We prove the claim by contradiction. Suppose by the contradiction that there exist sequences of $\eta_{i}\rightarrow 0$ and $r_{i}\rightarrow 0$ as $i\rightarrow +\infty$ such that
\begin{equation*}
\int_{\mathcal{W}_{r_{i}}(x_{0})} \phi_{i} F^{n}(\nabla u_{\epsilon})dx> 1-\eta_{i}.
\end{equation*}
for every $\phi_{i}\in C_{0}^{1}(\mathcal{W}_{r_{i}}(x_{0}) )$ and $\phi_{i}=1$ in $\mathcal{W}_{\frac{r_{i}}{2}}(x_{0})$.
Then
\begin{equation}\label{4-001}
\int_{\mathcal{W}_{\frac{r_{i}}{2}}(x_{0})} \phi_{i} F^{n}(\nabla u_{\epsilon})dx> 1-\eta_{i}.
\end{equation}
Taking $i\rightarrow\infty$, the left hand side of (\ref{4-001})converges to $0$. However, $1-\eta_{i}\rightarrow 1$, this contradiction leads to the claim. Since $u_{\epsilon}\rightarrow 0$ strongly in $L^{q}(\Omega)$ for any $q>1$, we may assume that
\begin{equation*}
\int_{\mathcal{W}_{r}(x_{0})}  F^{n}(\nabla (\phi u_{\epsilon}))dx\leq 1-\eta
\end{equation*}
provided $\epsilon$ is sufficient small. By (\ref{1-02}), $e^{\lambda_{n}{(\phi u_{\epsilon})}^{\frac{n}{n-1}}}$ is uniformly bounded in $L^{s}(\mathcal{W}_{r_{0}}(x_{0}))$ for some $s>1$ and $0<r_{0}<r$. Applying Lemma \ref{2-04}, $u_{\epsilon}$ is uniformly bounded in $\mathcal{W}_{\frac {r_{0}}2}(x_{0})$, which contradicts the fact that $M_{\epsilon}\rightarrow +\infty$ again. Therefore, $F^{n}(\nabla u_{\epsilon})dx\rightharpoonup \delta_{x_{0}}$ as $\epsilon\rightarrow 0$.
\end{proof}

Now we set
\begin{equation}\label{4-7}
r_{\varepsilon}^{n}=\lambda_{\varepsilon}\beta_{\epsilon}^{-1}M_{\varepsilon}^{-\frac{n}{n-1}}e^{-\alpha_{\epsilon}M_{\epsilon}^{\frac{n}{n-1}}}.
\end{equation}
Fixed any $\delta\in (0,\frac{\lambda_{n}}{2})$, by the expression of $r_{\epsilon}$ in (\ref{4-7}) and $\lambda_{\epsilon}$ in
(\ref{3-008}), we have
\begin{eqnarray*}
r_{\epsilon}^{n}exp\{\delta M_{\epsilon}^{\frac{n}{n-1}}\}&=&\lambda_{\epsilon}\beta_{\epsilon}^{-1}M_{\epsilon}^{-\frac{n}{n-1}}exp\{-\alpha_{\epsilon} M_{\epsilon}^{\frac{n}{n-1}}\}exp\{\delta M_{\epsilon}^{\frac{n}{n-1}}\}\\
&=&\beta_{\epsilon}^{-1}M_{\epsilon}^{-\frac{n}{n-1}}exp\{(\delta-\alpha_{\epsilon}) M_{\epsilon}^{\frac{n}{n-1}}\}\int_{\Omega}u_{\epsilon}^{\frac{n}{n-1}}exp\{\alpha_{\epsilon} u_{\epsilon}^{\frac{n}{n-1}}\}dx\\
&\leq& \beta_{\epsilon}^{-1}M_{\epsilon}^{-\frac{n}{n-1}}exp\{(2\delta-\alpha_{\epsilon}) M_{\epsilon}^{\frac{n}{n-1}}\}\int_{\Omega}u_{\epsilon}^{\frac{n}{n-1}}exp\{(\alpha_{\epsilon}-\delta) u_{\epsilon}^{\frac{n}{n-1}}\}dx\\
&\leq& C \beta_{\epsilon}^{-1}M_{\epsilon}^{-\frac{n}{n-1}}exp\{(2\delta-\alpha_{\epsilon}) M_{\epsilon}^{\frac{n}{n-1}}\}\\
&\rightarrow& 0
\end{eqnarray*}
as $\epsilon\rightarrow 0$. From above, we can easily get the fact that $\beta_{\epsilon}\rightarrow 1$, $\alpha_{\epsilon}\rightarrow \lambda_{n}$, $r_{\epsilon}\rightarrow 0$, $\gamma_{\epsilon}\rightarrow \alpha$ as $\epsilon\rightarrow 0$.

 Define the rescaling functions
\begin{eqnarray}\label{4-6}
v_{\epsilon}(x)&=&\frac{u_{\epsilon}(x_{\epsilon}+r_{\epsilon}x)}{M_{\epsilon}},\nonumber\\
w_{\epsilon}(x)&=&M_{\epsilon}^{\frac{1}{n-1}}(u_{\epsilon}(x_{\epsilon}+r_{\epsilon}x)-M_{\epsilon}),
\end{eqnarray}
where $v_{\epsilon}(x)$ and $w_{\epsilon}(x)$ are defined on $\Omega_{\epsilon}=\{x\in\mathbb{R}^{n}:~~x_{\epsilon}+r_{\epsilon}x\in \Omega\}$.

By a direct calculation we obtain that
\begin{equation*}
-div(F^{n-1}(\nabla v_{\epsilon})F_{\xi}(\nabla v_{\epsilon}))=\frac{v_{\epsilon}^{\frac{1}{n-1}}}{M_{\epsilon}^{n}}
e^{\alpha_{\epsilon}(u_{\epsilon}^{\frac{n}{n-1}}(x_{\epsilon}+r_{\epsilon}x)-M_{\epsilon}^{\frac{n}{n-1}})}
+r_{\epsilon}^{n}\gamma_{\epsilon}v_{\epsilon}^{n-1}\qquad \text{ in }\Omega_{\epsilon}  .
\end{equation*}

Since  $0\leq  v_{\epsilon}\leq 1$, $\frac{v_{\epsilon}^{\frac{1}{n-1}}}{M_{\epsilon}^{n}}
e^{\alpha_{\epsilon}(u_{\epsilon}^{\frac{n}{n-1}}(x_{\epsilon}+r_{\epsilon}x)-M_{\epsilon}^{\frac{n}{n-1}})}
+r_{\epsilon}^{n}\gamma_{\epsilon}v_{\epsilon}^{n-1}\rightarrow 0$  in  $B_{r}(0)$  for any   $r>0$, and  $\frac{v_{\epsilon}^{\frac{1}{n-1}}}{M_{\epsilon}^{n}}
e^{\alpha_{\epsilon}(u_{\epsilon}^{\frac{n}{n-1}}(x_{\epsilon}+r_{\epsilon}x)-M_{\epsilon}^{\frac{n}{n-1}})}
+r_{\epsilon}^{n}\gamma_{\epsilon}v_{\epsilon}^{n-1}$  is uniformly bounded in $ L^{\infty}(\overline{B_{r}(0)})$,   by Theorem 1 in \cite{T2}, $v_{\epsilon}$ is uniformly bounded in $C^{1,\alpha}(\overline{B_{\frac{r}{2}}(0)})$. By Ascoli-Arzela's theorem, we can find a subsequence $\epsilon_{j}\rightarrow 0$  such that  $v_{\epsilon_{j}} \rightarrow v$ in $C_{loc}^{1}(\mathbb{R}^{n})$,   where $v\in C^{1}(\mathbb{R}^{n})$  and satisfies
$$-div(F^{n-1}(\nabla v)F_{\xi}(\nabla v))=0\qquad \text{ in } \mathbb{R}^{n} .$$
Furthermore, we have  $0\leq v\leq 1$  and  $v(0)=1$.  The Liouville theorem (see \cite{HKM}) leads to $v\equiv 1$.

Also we have in $\Omega_{\epsilon}$
\begin{equation}\label{4-8}
-div(F^{n-1}(\nabla w_{\epsilon})F_{\xi}(\nabla w_{\epsilon}))=v_{\epsilon}^{\frac{1}{n-1}}e^{\alpha_{\epsilon}(u_{\epsilon}^{\frac{n}{n-1}}(x_{\epsilon}+r_{\epsilon}x)-M_{\epsilon}^{\frac{n}{n-1}})}~~~~ +r_{\epsilon}^{n}M_{\epsilon}\gamma_{\epsilon}u_{\epsilon}^{n-1}
.
\end{equation}
For any  $r>0$,  since  $0\leq u_{\epsilon}(x_{\epsilon}+r_{\epsilon}x)\leq M_{\epsilon}$, we have $-div(F^{n-1}(\nabla w_{\epsilon})F_{\xi}(\nabla w_{\epsilon}))=O(1)$  in  $B_{r}(0)$  for small $\epsilon$.  Then from Theorem 1 in \cite{T2} and Ascoli-Arzela's theorem, there exists $w\in C^{1}(\mathbb{R}^{n})$ such that $w_{\epsilon}$ converges to $w$ in $C_{loc}^{1}(\mathbb{R}^{n})$.  Therefore we have
\begin{eqnarray}
|u_{\epsilon}|^{\frac{n}{n-1}}(x_{\epsilon}+r_{\epsilon}x)-M_{\epsilon}^{\frac{n}{n-1}}&=&M_{\epsilon}^{\frac{n}{n-1}}(v_{\epsilon}^{\frac{n}{n-1}}(x)-1)\nonumber\\
&=&\frac{n}{n-1}w_{\epsilon}(x)(1+O((v_{\epsilon}(x)-1)^{2})).
\end{eqnarray}
By taking   $\epsilon\rightarrow 0$,  we know that  $w$  satisfies
\begin{equation}\label{4-9}
 -div(F^{n-1}(\nabla w)F_{\xi}(\nabla w))=e^{\frac{n}{n-1}\lambda_{n} w}
\end{equation}
in the distributional sense. We also have the facts   $w(0)=0= \max_{x\in \R^{n}}w(x)$.   Moreover, for any $R>0$, we have

\begin{eqnarray}\label{4-10}
1&\geq&\lim_{\epsilon\rightarrow 0}\int_{\mathcal{W}_{r_{\epsilon} R}(x_{\epsilon})}\frac{u_{\epsilon}^{\frac{n}{n-1}}}{\lambda_{\epsilon}}e^{\alpha_{\epsilon}|u_{\epsilon}|^{\frac{n}{n-1}}}dx
\nonumber\\
&=&\lim_{\epsilon\rightarrow 0}\int_{\mathcal{W}_{ R}(0)}v_\epsilon^{\frac{n}{n-1}}e^{\alpha_{\epsilon}(u_\epsilon^{\frac{n}{n-1}}(x_{\epsilon}+r_{\epsilon}x)-M_{\epsilon}^{\frac{n}{n-1}})}dx\nonumber\\
&=&\int_{\mathcal{W}_{R}(0)}e^{\frac{n}{n-1}\lambda_{n} w}dx.
\end{eqnarray}
Taking  $ R\rightarrow+\infty$,  we have  $\int_{\R^{n}}e^{\frac{n}{n-1}\lambda_{n} w}dx\leq 1$.

On the other hand, we claim
\begin{equation}\label{4-11}
\int_{\R^{n}}e^{\frac{n}{n-1}\lambda_{n} w}dx\geq 1.
\end{equation}
Actually we can prove it by level-set-method. For  $t\in \R$, let $\Omega_{t}=\{x\in\Omega|w(x)>t\}$ and $\mu(t)=|\Omega_{t}|$.  By the divergence theorem,
\begin{eqnarray}
\int_{\Omega_{t}}-div(F^{n-1}(\nabla w)F_{\xi}(\nabla w))dx&=&\int_{\partial\Omega_{t}}F^{n-1}(\nabla w)<F_{\xi}(\nabla w),\frac{\nabla w}{|\nabla w|}>ds\nonumber\\
&=&\int_{\partial\Omega_{t}}\frac{F^{n}(\nabla w)}{|\nabla w|}ds.\nonumber
\end{eqnarray}
By using the isoperimetric inequality (\ref{2-03}) and the co-area formula (\ref{2-02}), it follows from H\"{o}lder inequality that
\begin{eqnarray}\label{4-12}
n\kappa_{n}^{\frac{1}{n}}\mu(t)^{1-\frac{1}{n}} &\leq & P_F(\Omega_t)
 = \int_{\partial\Omega_{t}}\frac{F(\nabla w)}{|\nabla w|}ds \nonumber \\
&\leq &(\int_{\partial\Omega_{t}}\frac{F^{n}(\nabla w)}{|\nabla w|}ds)^{\frac{1}{n}}(\int_{\partial\Omega_{t}}\frac{1}{|\nabla w|}ds)^{1-\frac{1}{n}}\nonumber\\
&=&(\int_{\Omega_{t}}e^{\frac{n}{n-1}\lambda_{n} w}dx)^{\frac{1}{n}}(-\mu^{\prime}(t))^{1-\frac{1}{n}}.
\end{eqnarray}
Hence
\begin{eqnarray}
\int_{\R^{n}}e^{\frac{n}{n-1}\lambda_{n} w}dx&=&\frac{n}{n-1}\lambda_{n}\int_{-\infty}^{\max w}e^{\frac{n}{n-1}\lambda_{n} t}\mu(t)dt\nonumber\\
&\leq& \frac{n}{n-1}\lambda_{n}\int_{-\infty}^{\max w}e^{\frac{n}{n-1}\lambda_{n}t}\frac{-\mu^{\prime}(t)}{(n\kappa_{n}^{\frac{1}{n}})^{\frac{n}{n-1}}}(\int_{\Omega_{t}}e^{\frac{n}{n-1}\lambda_{n} w}dx)^{\frac{1}{n-1}}dt\nonumber\\
&=&\int_{-\infty}^{\max w}\frac{d}{dt}(\int_{\Omega_{t}}e^{\frac{n}{n-1}\lambda_{n} w}dx)^{\frac{n}{n-1}}dt=(\int_{\R^{n}}e^{\frac{n}{n-1}\lambda_{n} w}dx)^{\frac{n}{n-1}},\nonumber
\end{eqnarray}
which implies the claim.

Thus we get that $\int_{\R^{n}}e^{\frac{n}{n-1}\lambda_{n} w}dx=1$, which implies that the equality holds in the above iso-perimetric inequality. Therefore $\Omega_{t}$  must be a wulff ball. In other words, $w$ is radial symmetric with respect to  $F^o(x)$.  We can immediately get
\begin{equation}\label{4-13}
w(r)=-\frac{n-1}{\lambda_{n}}\log(1+\kappa_{n}^{\frac{1}{n-1}} r^{\frac{n}{n-1}})
\end{equation} where $r=F^{o}(x)$.
\begin{lm} \label{4-14}
 For any $L>1$, we set $u_{\epsilon,L}= min\{u_{\epsilon},\frac{M_{\epsilon}}{L}\}$. Then we have
\begin{equation}
   \limsup_{\epsilon\rightarrow 0}\int_{\Omega}F^{n}(\nabla u_{\epsilon,L})dx\leq\frac{1}{L}.\nonumber
\end{equation}
\end{lm}
\begin{proof}
We chose $(u_{\epsilon}-\frac{M_{\epsilon}}{L})^{+}$ as a test function of (\ref{3-008}) to get
\begin{eqnarray}\label{4-15}
& & -\int_{\Omega}(u_{\epsilon}-\frac{M_{\epsilon}}{L})^{+}div(F^{n-1}(\nabla u_{\epsilon})F_{\xi}(\nabla u_{\epsilon}))dx\nonumber \\
&= & \int_{\Omega}(u_{\epsilon}-\frac{M_{\epsilon}}{L})^{+}[\frac{\beta_{\epsilon}u_{\epsilon}^{\frac{1}{n-1}}}{\lambda_{\epsilon}}e^{\alpha_{\epsilon}|u_{\epsilon}|^{\frac{n}{n-1}}}+\gamma_{\epsilon}u_{\epsilon}^{n-1}]dx.
\end{eqnarray}

For any $R>0$, the estimation of the right hand side of (\ref{4-15}) is
\begin{eqnarray}\label{4-16}
& & \int_{\Omega}(u_{\epsilon}-\frac{M_{\epsilon}}{L})^{+}[\frac{\beta_{\epsilon}u_{\epsilon}^{\frac{1}{n-1}}}{\lambda_{\epsilon}}e^{\alpha_{\epsilon}|u_{\epsilon}|^{\frac{n}{n-1}}}+
\gamma_{\epsilon}u_{\epsilon}^{n-1}]dx\nonumber\\
&\geq&\int_{\mathcal{W}_{r_{\epsilon} R}(x_{\epsilon})}(u_{\epsilon}-\frac{M_{\epsilon}}{L})^{+}\frac{\beta_{\epsilon}u_{\epsilon}^{\frac{1}{n-1}}}{\lambda_{\epsilon}}e^{\alpha_{\epsilon}|u_{\epsilon}|^{\frac{n}{n-1}}}dx+o_{\epsilon}(1)\nonumber\\
&=&\int_{\mathcal{W}_{ R}(0)}(u_{\epsilon}(x_{\epsilon}+r_{\epsilon}x)-\frac{M_{\epsilon}}{L})^{+}\frac{r_{\epsilon}^{n}\beta_{\epsilon}u_{\epsilon}^{\frac{1}{n-1}}(x_{\epsilon}+r_{\epsilon}x)}{\lambda_{\epsilon}}e^{\alpha_{\epsilon}|u_{\epsilon}|^{\frac{n}{n-1}}(x_{\epsilon}+r_{\epsilon}x)}dx+o_{\epsilon}(1)\nonumber\\
&=&\int_{\mathcal{W}_{ R}(0)}(v_{\epsilon}-\frac{1}{L})^{+}\beta_{\epsilon}v_{\epsilon}^{\frac{1}{n-1}}e^{\alpha_{\epsilon}(|u|^{\frac{n}{n-1}}(x_{\epsilon}+r_{\epsilon}x)-M_{\epsilon}^{\frac{n}{n-1}})}dx+o_{\epsilon}(1)\nonumber\\
&\rightarrow&\int_{\mathcal{W}_{R}(0)}(1-\frac{1}{L})e^{\frac{n}{n-1}\lambda_{n} w}dx.
\end{eqnarray}
In virtue of the  divergence theorem and Lemma \ref{2-01}, the estimation of the left hand side of (\ref{4-15}) is
\begin{eqnarray}\label{4-17}
&-&\int_{\Omega}(u_{\epsilon}-\frac{M_{\epsilon}}{L})^{+}div(F^{n-1}(\nabla u_{\epsilon})F_{\xi}(\nabla u_{\epsilon}))dx\nonumber\\
&=&-\int_{\Omega}(u_{\epsilon}-\frac{M_{\epsilon}}{L})^{+}div(F^{n-1}(\nabla (u_{\epsilon}-\frac{M_{\epsilon}}{L})^{+})F_{\xi}(\nabla (u_{\epsilon}-\frac{M_{\epsilon}}{L})^{+}))dx\nonumber\\
&=&\int_{\Omega}F^{n}(\nabla (u_{\epsilon}-\frac{M_{\epsilon}}{L})^{+})dx.
\end{eqnarray}
Putting (\ref{4-15}) (\ref{4-16})(\ref{4-17}) together, and taking $R\rightarrow \infty $,  we obtain
\begin{equation}
\int_{\Omega}F^{n}(\nabla (u_{\epsilon}-\frac{M_{\epsilon}}{L})^{+})dx\geq1-\frac{1}{L}.\nonumber
\end{equation}
Noticing that
\begin{equation}
\int_{\Omega}F^{n}(\nabla u_{\epsilon})dx=\int_{\Omega}F^{n}(\nabla u_{\epsilon,L})dx+\int_{\Omega}F^{n}(\nabla (u_{\epsilon}-\frac{M_{\epsilon}}{L})^{+})dx.\nonumber
\end{equation}
Thus the conclusion can be obtained  due to the fact   $\int_{\Omega}F^{n}(\nabla u_{\epsilon})dx=1$.
\end{proof}

\begin{rem} From Lemma \ref{4-14}, applying $L^{\frac{n-1}{n}}u_{\epsilon,L}$ to inequality (\ref{1-02}), we get
\begin{equation}\label{4-18}
\int_{\Omega}e^{\lambda_{n}L|u_{\epsilon,L}|^{\frac{n}{n-1}}}dx\leq C<+\infty.
\end{equation}
\end{rem}

\
\

 For any   $L>1$,
 since  $e^{\alpha_{\epsilon}\frac{L+1}{2}|u_{\epsilon,L}|^{\frac{n}{n-1}}}$  is uniformly bounded in  $L^{1}(\Omega)$,  then
by (\ref{1-02}) $e^{\alpha_{\epsilon} |u_{\epsilon,L}|^{\frac{n}{n-1}}}$  is uniformly bounded in  $L^{q}(\Omega)$  for some  $q>1$.  Due to  $u_{\epsilon,L}$  converges to  $0$  almost everywhere in  $\Omega$, it implies that  $e^{\alpha_{\epsilon} |u_{\epsilon,L}|^{\frac{n}{n-1}}}$  converges to $1$ in $L^{1}(\Omega)$.  Thus we have
\begin{eqnarray}
\lim_{\epsilon\rightarrow 0}\int_{\{Lu_{\epsilon}\leq M_{\epsilon}\}}e^{\alpha_{\epsilon}|u_{\epsilon}|^{\frac{n}{n-1}}}dx= \lim_{\epsilon\rightarrow 0}\int_{\Omega}e^{\alpha_{\epsilon}|u_{\epsilon,L}|^{\frac{n}{n-1}}}dx=|\Omega|. \nonumber
\end{eqnarray}
Hence
\begin{eqnarray}
& &\lim_{\epsilon\rightarrow 0}\int_{\Omega}e^{\alpha_{\epsilon}|u_{\epsilon}|^{\frac{n}{n-1}}}dx\nonumber\\
&=&\lim_{\epsilon\rightarrow 0}\int_{\{Lu_{\epsilon}\leq M_{\epsilon}\}}e^{\alpha_{\epsilon} |u_{\epsilon}|^{\frac{n}{n-1}}}dx+\lim_{\epsilon\rightarrow 0}\int_{\{Lu_{\epsilon}> M_{\epsilon}\}}e^{\alpha_{\epsilon} |u_{\epsilon}|^{\frac{n}{n-1}}}dx\nonumber\\
&\leq& |\Omega|+\lim_{\epsilon\rightarrow 0}\frac{\lambda_{\epsilon}L^{\frac{n}{n-1}}}{M_{\epsilon}^{\frac{n}{n-1}}}\int_{\{Lu_{\epsilon}> M_{\epsilon}\}}\frac{u_{\epsilon}^{\frac{n}{n-1}}}{\lambda_{\epsilon}}e^{\alpha_{\epsilon} |u_{\epsilon}|^{\frac{n}{n-1}}}dx\nonumber\\
&\leq & |\Omega|+\lim_{\epsilon\rightarrow 0}\frac{\lambda_{\epsilon}L^{\frac{n}{n-1}}}{M_{\epsilon}^{\frac{n}{n-1}}}.\nonumber
\end{eqnarray}
Taking $L\rightarrow 1$, we get
\begin{equation}\label{4-19}
\lim_{\epsilon\rightarrow 0}\int_{\Omega}e^{\alpha_{\epsilon} |u_{\epsilon}|^{\frac{n}{n-1}}}dx\leq|\Omega|+\limsup_{\epsilon\rightarrow 0}\frac{\lambda_{\epsilon}}{M_{\epsilon}^{\frac{n}{n-1}}}.
\end{equation}
Noticing that  $\lim_{\epsilon\rightarrow 0}\int_{\Omega}e^{\alpha_{\epsilon}|u_{\epsilon}|^{\frac{n}{n-1}}}dx>|\Omega|$,  we have
 \begin{equation}\label{4-20}
 \lim_{\epsilon\rightarrow 0}\frac{M_{\epsilon}}{\lambda_{\epsilon}}=0.
 \end{equation}

\
\

The following Lemma will be used in Section $5$.
\begin{lm}\label{4-21}
\begin{equation*}
\lim_{\epsilon\rightarrow 0}\int_{\Omega}e^{\alpha_{\epsilon}|u_{\epsilon}|^{\frac{n}{n-1}}}dx= |\Omega|+\lim_{R\rightarrow +\infty}\limsup_{\epsilon\rightarrow 0}\int_{\mathcal{W}_{Rr_{\epsilon}}(x_{\epsilon})}e^{\alpha_{\epsilon} |u_{\epsilon}|^{\frac{n}{n-1}}}dx.
\end{equation*}
\end{lm}

\begin{proof}
On one hand,
\begin{eqnarray}\label{4-22}
& &\limsup_{\epsilon\rightarrow 0}\int_{\mathcal{W}_{Rr_{\epsilon}}(x_{\epsilon})}e^{\alpha_{\epsilon} |u_{\epsilon}|^{\frac{n}{n-1}}}dx\nonumber\\
&\leq &\limsup_{\epsilon\rightarrow 0}\int_{\Omega}e^{\alpha_{\epsilon} |u_{\epsilon}|^{\frac{n}{n-1}}}dx-\liminf_{\epsilon\rightarrow 0}\int_{\Omega\backslash \mathcal{W}_{Rr_{\epsilon}}(x_{\epsilon})}e^{\alpha_{\epsilon} |u_{\epsilon}|^{\frac{n}{n-1}}}dx\nonumber\\
&\leq &\limsup_{\epsilon\rightarrow 0}\int_{\Omega}e^{\alpha_{\epsilon} |u_{\epsilon}|^{\frac{n}{n-1}}}dx-|\Omega|.
\end{eqnarray}
On the other hand,
\begin{equation*}
\int_{\mathcal{W}_{Rr_{\epsilon}}(x_{\epsilon})}e^{\alpha_{\epsilon} |u_{\epsilon}|^{\frac{n}{n-1}}}dx=\frac{\lambda_{\epsilon}}
{M_{\epsilon}^{\frac{n}{n-1}}}(\int_{\mathcal{W}_{R}(0)}e^{\frac{n}{n-1}\lambda_{n}w}dx+o_{\epsilon}(1)),
\end{equation*}
which gives
\begin{equation}\label{4-23}
\lim_{R\rightarrow +\infty}\limsup_{\epsilon\rightarrow 0}\int_{\mathcal{W}_{Rr_{\epsilon}}(x_{\epsilon})}e^{\alpha_{\epsilon} |u_{\epsilon}|^{\frac{n}{n-1}}}dx=\limsup_{\epsilon\rightarrow 0}\frac{\lambda_{\epsilon}}{M_{\epsilon}^{\frac{n}{n-1}}}.
\end{equation}
Combining  (\ref{4-19}), (\ref{4-22}),  (\ref{4-23}) and  Lemma \ref{3-10},  we get the result.
\end{proof}

Now we  claim that
\begin{equation}\label{4-24}
\lim_{\epsilon\rightarrow 0}\int_{\Omega}\frac{M_{\epsilon}}{\lambda_{\epsilon}}\beta_{\epsilon}u_{\epsilon}^{\frac{1}{n-1}}e^{\alpha_{\epsilon} |u_{\epsilon}|^{\frac{n}{n-1}}}dx=1.
\end{equation}
To this purpose, we denote $\varphi_{\epsilon}=\frac{M_{\epsilon}}{\lambda_{\epsilon}}\beta_{\epsilon}u_{\epsilon}^{\frac{1}{n-1}}e^{\alpha_{\epsilon} |u_{\epsilon}|^{\frac{n}{n-1}}}$. Clearly
\begin{equation*}
\int_{\Omega}\varphi_{\epsilon} dx=\int_{\{Lu_{\epsilon}< M_{\epsilon}\}}\varphi_{\epsilon} dx+\int_{\{Lu_{\epsilon}\geq M_{\epsilon}\}\backslash \mathcal{W}_{r_{\epsilon}R}(x_{\epsilon})}\varphi_{\epsilon} dx+\int_{\mathcal{W}_{r_{\epsilon} R}(x_{\epsilon})}\varphi_{\epsilon} dx.
\end{equation*}
We estimate the three integrals on the right hand side respectively. By (\ref{4-20}) and Lemma \ref{4-000} we have
\begin{eqnarray}\label{4-25}
0 \leq \int_{\{Lu_{\epsilon}< M_{\epsilon}\}}\varphi_{\epsilon} dx&=&\frac{M_{\epsilon}\beta_{\epsilon}}{\lambda_{\epsilon}}\int_{\{Lu_{\epsilon}< M_{\epsilon}\}}u_{\epsilon}^{\frac{1}{n-1}}e^{\alpha_{\epsilon} |u_{\epsilon}|^{\frac{n}{n-1}}}dx\nonumber\\
&\leq&\frac{M_{\epsilon}\beta_{\epsilon}}{\lambda_{\epsilon}}\int_{\Omega}u_{\epsilon,L}^{\frac{1}{n-1}}e^{\alpha_{\epsilon} |u_{\epsilon,L}|^{\frac{n}{n-1}}}dx\nonumber\\
&=&o_{\epsilon}(1)O(\frac 1L).
\end{eqnarray}
 Moreover for any $R>0$, we have
\begin{eqnarray}\label{4-26}
\int_{\{Lu_{\epsilon}\geq M_{\epsilon}\}\backslash \mathcal{W}_{r_{\epsilon} R}(x_{\epsilon})}\varphi_{\epsilon}dx &\leq& \int_{\{Lu_{\epsilon}\geq M_{\epsilon}\}\backslash \mathcal{W}_{r_{\epsilon} R}(x_{\epsilon})}\frac{L\beta_{\epsilon}}{\lambda_{\epsilon}}u_{\epsilon}^{\frac{n}{n-1}}e^{\alpha_{\epsilon} |u_{\epsilon}|^{\frac{n}{n-1}}}dx\nonumber\\
&\leq& \int_{\Omega\backslash \mathcal{W}_{r_{\epsilon} R}(x_{\epsilon})}\frac{L\beta_{\epsilon}}{\lambda_{\epsilon}}u_{\epsilon}^{\frac{n}{n-1}}e^{\alpha_{\epsilon} |u_{\epsilon}|^{\frac{n}{n-1}}}dx\nonumber\\
&\rightarrow&L(1-\int_{\mathcal{W}_{ R}(0)}e^{\frac{n}{n-1}\lambda_{n} w}dx),
\end{eqnarray}
and
\begin{eqnarray}\label{4-27}
\int_{\mathcal{W}_{r_{\epsilon} R}(x_\epsilon)}\varphi_{\epsilon} dx&=&\int_{\mathcal{W}_{ R}(0)}\beta_{\epsilon}v_{\epsilon}^{\frac{1}{n-1}}e^{\alpha_{\epsilon} (|u_{\epsilon}|^{\frac{n}{n-1}}-M_{\epsilon}^{\frac{n}{n-1}})}dx\nonumber\\
&\rightarrow&\int_{\mathcal{W}_{ R}(0)}e^{\frac{n}{n-1}\lambda_{n} w}dx.
\end{eqnarray}
Putting (\ref{4-25}) (\ref{4-26}) (\ref{4-27}) together and taking  $\epsilon\rightarrow 0$  first, then letting $R\rightarrow   \infty$, we conclude (\ref{4-24}).

In the similar way, we also can obtain that
\begin{equation}\label{4-244}
\lim_{\epsilon\rightarrow 0}\int_{\Omega}\frac{\beta_{\epsilon}M_{\epsilon}}{\lambda_{\epsilon}}u_{\epsilon}^{\frac{1}{n-1}}e^{\alpha_{\epsilon} |u_{\epsilon}|^{\frac{n}{n-1}}}\phi(x)dx=\phi(x_{0})
\end{equation}for any $\phi(x)\in C_c^0(\Omega)$.

The following phenomenon was first discovered by Brezis and Merle \cite{BM}, developed later by Struwe \cite{S1}. We deduce the new version involving n-Finsler-Laplacian.
\begin{lm}\label{4-49}
Let $\{f_{\epsilon}\}$ be a uniformly bounded sequence of functions in $L^{1}(\Omega)$, and $\{\psi_{\epsilon}\}\subset C^{1}(\overline \Omega)\cap W_{0}^{1,n}(\Omega)$ satisfy
\begin{equation}\label{4-28}
-div(F^{n-1}(\nabla \psi_{\epsilon})F_{\xi}(\nabla  \psi_{\epsilon}))=f_{\epsilon}+\alpha\psi_{\epsilon}|\psi_{\epsilon}|^{n-2} \qquad in ~~~~\Omega.
\end{equation}
where $0\leq \alpha<\lambda_{1}(\Omega)$ is a constant. Then for any $1<q<n$, we have $||\nabla \psi_{\epsilon}||_{L^{q}(\Omega)}\leq C$ for some constant $C$  depending only on $q,n,\Omega$ and the upper bound of $||f_{\epsilon}||_{L^{1}(\Omega)}$.
\end{lm}
\begin{proof}
When $\alpha=0$. We use an argument of M. Struwe to prove that $||\nabla \psi_{\epsilon}||_{L^{q}(\Omega)}\leq C||f_{\epsilon}||_{L^{1}(\Omega)}$ for some constant $C$  depending only on $q,n,\Omega$. Without loss of generality, we assume $||f_{\epsilon}||_{L^{1}(\Omega)}=1$. For $t\geq 1$, denote $\psi_{\epsilon}^{t}=\min\{\psi_{\epsilon}^{+},t \}$, where $\psi_{\epsilon}^{+}$ is a positive part of $\psi_{\epsilon}$. Testing Eq.(\ref{4-28}) with $\psi_{\epsilon}^{t}$, we have $\int_{\Omega}F^{n}(\nabla \psi_{\epsilon}^{t})dx\leq\int_{\Omega} |f_{\epsilon}|\psi_{\epsilon}^{t}\leq t$. Assume $|\Omega|=|\mathcal{W}_{d}|$, where $\mathcal{W}_{d}=\{x\in \mathbb{R}^{n}: F^{0}(x)\leq d \}$.

Let $\psi_{\epsilon}^{\star}$ be the nonincreasing  rearrangement of $\psi_{\epsilon}^{t}$, and $|\mathcal{W}_{\rho}|=|\{x\in \mathcal{W}_{d}: \psi_{\epsilon}^{\star}\geq t\}|$. It is known that $||F(\nabla \psi_{\epsilon}^{\star})||_{L^{n}(\mathcal{W}_{d})}\leq ||F(\nabla \psi_{\epsilon}^{t})||_{L^{n}(\Omega)}$, and we have
\begin{equation}\label{4-29}
\inf_{\phi\in W_{0}^{1,n}(\mathcal{W}_{d}),\phi|_{\mathcal{W}_{\rho}}=t}\int_{\mathcal{W}_{d}}F^{n}(\nabla \phi)dx\leq \int_{\mathcal{W}_{d}}F^{n}(\nabla \psi_{\epsilon}^{\star})dx\leq t.
\end{equation}
The above infimum can be attained by
\begin{equation*}
\phi_{1}(x)=
\left\{
\begin{array}{lr}
t\log\frac{d}{F^{0}(x)}/\log\frac{d}{\rho}\qquad &in~~~\mathcal{W}_{d}\backslash \mathcal{W}_{\rho},\\
t  \qquad &in~~\mathcal{W}_{\rho}.
\end{array}
\right.
\end{equation*}
Calculating $||F(\nabla \phi_{1})||_{L^{n}(\mathcal{W}_{d})}^{n}$, we have by (\ref{4-29}), $\rho\leq de^{-C_{1}t}$ for some constant $C_{1}>0$. Hence
\begin{equation*}
|\{x\in\Omega: \psi_{\epsilon}\geq t\}|=|\mathcal{W}_{\rho}|\leq \kappa_{n}d^{n}e^{-nC_{1}t}.
\end{equation*}
For any $0<\delta<nC_{1}$,
\begin{eqnarray*}
\int_{\Omega}e^{\delta\psi_{\epsilon}^{+}}dx&\leq& e^{\delta}|\Omega|+\sum_{m=1}^{\infty}e^{(m+1)\delta}|\{x\in\Omega:~~m\leq \psi_{\epsilon}\leq m+1\}|\\
&\leq&e^{\delta}|\Omega|+\kappa_{n}d^{n}e^{\delta}\sum_{m=1}^{\infty}e^{-(nC_{1}-\delta)m}\leq C_{2}
\end{eqnarray*}
for some constant $C_{2}$. Testing Eq.(\ref{4-28}) with $\log\frac{1+2\psi_{\epsilon}^{+}}{1+\psi_{\epsilon}^{+}}$, we have
\begin{equation*}
\int_{\Omega}\frac{F^{n}(\nabla\psi_{\epsilon}^{+})}{(1+\psi_{\epsilon}^{+})(1+2\psi_{\epsilon}^{+})}dx\leq \log 2.
\end{equation*}
By the Young inequality, we have for any $1<q<n$,
\begin{eqnarray*}
\int_{\Omega}F^{q}(\nabla \psi_{\epsilon}^{+})dx&\leq &\int_{\Omega}\frac{F^{n}(\nabla\psi_{\epsilon}^{+})}{(1+\psi_{\epsilon}^{+})(1+2\psi_{\epsilon}^{+})}dx
+\int_{\Omega}((1+\psi_{\epsilon}^{+})(1+2\psi_{\epsilon}^{+}))^{\frac{q}{n-q}}dx\\
&\leq& C_{3}(1+\int_{\Omega}e^{\delta\psi_{\epsilon}^{+}}dx)\leq C_{4},
\end{eqnarray*}
for some constants $C_{3}$ and $C_{4}$ depending only on $q,~~n$ and $\Omega$. Let $\psi_{\epsilon}^{-}$ be the negative part of $\psi_{\epsilon}$. Similarly, we have $\int_{\Omega}F^{q}(\nabla \psi_{\epsilon}^{-})dx\leq C_{5}$ for some constant $C_{5}$ depending only on $q,~~n$ and $\Omega$. Then by Lemma \ref{2-01}, the lemma holds.

When $\alpha\in (0,\lambda_{1}(\Omega))$. Suppose $\psi_{\epsilon}$ is unbounded in $L^{n-1}(\Omega)$. Then their exist a subsequence $\{\epsilon_{j}\}$ such that $||\psi_{\epsilon_{j}}||_{L^{n-1}(\Omega)}\rightarrow +\infty$ as $j\rightarrow +\infty$. Let $w_{\epsilon_{j}}=\psi_{\epsilon_{j}}/||\psi_{\epsilon_{j}}||_{L^{n-1}(\Omega)}$. Then we have $||w_{\epsilon_{j}}||_{L^{n-1}(\Omega)}=1$, and $-Q_{n}(w_{\epsilon_{j}})$ is bounded in $L^{1}(\Omega)$. Hence $w_{\epsilon_{j}}$ is bounded in $W_{0}^{1,q}(\Omega)$ for any $0<q<n$. Assume $w_{\epsilon_{j}}$ converges to $w$ weakly in $W_{0}^{1,q}(\Omega)$ and strongly in $L^{n-1}(\Omega)$.It can be easily derived that $w$ is a weak solution of $-Q_{n}u=\alpha w|w|^{n-2}$ in $\Omega$. Since $0<\alpha<\lambda_{1}(\Omega)$, $w$ must be zero. On the other hand, $||w_{\epsilon_{j}}||_{L^{n-1}(\Omega)}=1$ leads to $||w||_{L^{n-1}(\Omega)}=1$, contradiction. Therefore $\psi_{\epsilon}$ must be bounded in $L^{n-1}(\Omega)$. It implies $f_{\epsilon}+\alpha\psi_{\epsilon}|\psi_{\epsilon}|^{n-2}\in L^{1}(\Omega) $.
Then, for any $1<q<n$, there exist a constant $C$ depending only on $q,n,\Omega$, and the upper bounded of $||f_{\epsilon}||_{L^{1}(\Omega)}$ such that $||\nabla\psi_{\epsilon}||_{L^{q}(\Omega)}\leq C$. Thus the proof is finished.
\end{proof}

The following lemma reveals how $u_{\epsilon}$ converges away from $x_{0}$.
\begin{lm}\label{4-30}
$M_{\epsilon}^{\frac{1}{n-1}}u_{\epsilon} \rightharpoonup G_{\alpha}$ weakly in $W_{0}^{1,q}(\Omega)$ for any $1<q<n$, where $G_{\alpha}$ is a Green function satisfying
\begin{equation}\label{4-31}
\left\{
\begin{array}{llc}
-div(F^{n-1}(\nabla G_{\alpha})F_{\xi}(\nabla G_{\alpha}))=\delta_{x_{0}}+\alpha G_{\alpha}^{n-1}\qquad & \text { in }~~~\Omega,\\
G_{\alpha}=0            \qquad    & \text { on }~~~\partial\Omega.
\end{array}
\right.
\end{equation}
Furthermore,
$M_{\epsilon}^{\frac{1}{n-1}}u_{\epsilon} \rightarrow G_{\alpha}$ in $C^{1}(\overline{\Omega'})$ for any domain $\Omega'\subset\subset\overline\Omega\backslash\{x_{0}\}$.
\end{lm}
\begin{proof}
By Eq.(\ref{3-008}), we have

\begin{equation}\label{4-32}
-Q_{n}(M_{\epsilon}^{\frac{1}{n-1}}u_{\epsilon})=\frac{M_{\epsilon}\beta_{\epsilon}u_{\epsilon}^{\frac{1}{n-1}}}{\lambda_{\epsilon}}
e^{\alpha_{\epsilon}|u_{\epsilon}|^{\frac{n}{n-1}}}+\gamma_{\epsilon}M_{\epsilon}u_{\epsilon}^{n-1}.
\end{equation}
Due to (\ref{4-24}) and Lemma \ref{4-49}, we obtain  that $M_{\epsilon}^{\frac{1}{n-1}}u_{\epsilon}$ is uniformly bounded in $W_{0}^{1,q}(\Omega)$ for any $1<q<n$. Assume $M_{\epsilon}^{\frac{1}{n-1}}u_{\epsilon} \rightharpoonup G_{\alpha}$ weakly in $W_{0}^{1,q}(\Omega)$. Testing Eq.(\ref{4-32}) with $\phi\in C_{0}^{\infty}(\Omega)$, we have by (\ref{4-244})
\begin{eqnarray*}
-\int_{\Omega}\phi Q_{n}(M_{\epsilon}^{\frac{1}{n-1}}u_{\epsilon})dx&=&\int_{\Omega}\phi \frac{M_{\epsilon}\beta_{\epsilon}u_{\epsilon}^{\frac{1}{n-1}}}{\lambda_{\epsilon}}e^{\alpha_{\epsilon}|u_{\epsilon}|^{\frac{n}{n-1}}}dx
+\gamma_{\epsilon}\int_{\Omega}\phi M_{\epsilon}u_{\epsilon}^{n-1}dx \\
&\rightarrow&\phi(x_{0})+\alpha\int_{\Omega}\phi G_{\alpha}^{n-1}dx.
\end{eqnarray*}
Hence
\begin{equation*}
\int_{\Omega}\nabla \phi F^{n-1}(\nabla G_{\alpha})F_{\xi}(\nabla G_{\alpha})dx=\phi(x_{0}) +\alpha G_{\alpha}^{n-1},
\end{equation*}
in the sense of measure and whence (\ref{4-31}) holds.

 For any fixed small $\delta$, we choose a cut-off function $\xi(x)\in C_{0}^{\infty}(\Omega\backslash \mathcal{W}_{\delta}(x_{0}))$ such that $\xi(x)=1$ on $\Omega\backslash \mathcal{W}_{3\delta}(x_{0})$. By Lemma \ref{4-000} we get $\int_{\Omega}F^{n}(\nabla(\xi u_{\epsilon}))dx\rightarrow 0$ as $\epsilon\rightarrow 0$. Then $e^{(\xi u_{\epsilon})^{\frac{n}{n-1}}}$ is bounded in $L^{s}(\Omega\backslash \mathcal{W}_{\delta}(x_{0}))$ for any $s>1$. In particular, $e^{u_{\epsilon}^{\frac{n}{n-1}}}$ is bounded in $L^{s}(\Omega\backslash \mathcal{W}_{3\delta}(x_{0}))$. Since $M_{\epsilon}^{\frac{1}{n-1}}u_{\epsilon}$ is bounded in $L^{q}(\Omega)$ for any $q>1$, H$\ddot{o}$lder inequality implies that $\frac{M_{\epsilon}u_{\epsilon}^{\frac{1}{n-1}}}{\lambda_{\epsilon}}e^{\alpha_{\epsilon}|u_{\epsilon}|^{\frac{n}{n-1}}}$ is uniformly bounded in $L^{s_{0}}(\Omega\backslash \mathcal{W}_{3\delta}(x_{0}))$. From the proof of the Lemma \ref{4-49} and $0\leq\gamma_{\epsilon}<\lambda_{1}(\Omega)$, we have $\gamma_{\epsilon}M_{\epsilon}u_{\epsilon}^{n-1}\in L^{1}(\Omega\backslash \mathcal{W}_{3\delta}(x_{0}))$. Then by Lemma \ref{2-04}, we have $$||M_{\epsilon}^{\frac{1}{n-1}}u_{\epsilon}||_{L^{\infty}(\Omega\backslash \mathcal{W}_{3\delta}(x_{0}))}<C.$$

 Theorem 1 in \cite{T2} and Ascoli-Arzela's theorem,  we have $M_{\epsilon}^{\frac{1}{n-1}}u_{\epsilon}$ converges to $G_{\alpha}$ in $C_{loc}^{1}(\Omega\backslash \mathcal{W}_{4\delta}(x_{0}))$.
\end{proof}

\begin{lm}\label{4-66}
 Asymptotic representation of Green function $G_{\alpha}$ is
\begin{equation}\label{4-39}
G_{\alpha}=-\frac{1}{(n\kappa_{n})^{\frac{1}{n-1}}}\log F^{o}(x-x_{0})+C_{G}+\psi(x)
\end{equation}
where $C_{G}$ is a constant, $\psi(x_{0})=0$ and $\psi(x)\in C^{0}(\overline \Omega)\cap C_{loc}^{1}(\Omega\backslash \{x_{0}\})$ such that $\lim_{x\rightarrow x_{0}}F^{o}(x-x_{0})\nabla \psi(x)=0$.
\end{lm}
Up to now, we have described the convergence behavior of $u_{\epsilon}$ near $x_{0}$ and away from $x_{0}$ when the concentration point $x_{0}$ in the interior of $\Omega$.

\
\

\noindent \textbf{Case 2.} $x_{0}$ lies on $\partial \Omega$.
\begin{lm}\label{4-33}
Let $d_{\epsilon}=dist(x_{\epsilon},\partial \Omega)$, and $r_{\epsilon}$ be defined in (\ref{4-7}). There holds $r_{\epsilon}/d_{\epsilon}\rightarrow 0$.
\end{lm}
\begin{proof}
Suppose not, there exists $R>0$ such that $d_{\epsilon}\leq Rr_{\epsilon}$. Take some $y_{\epsilon}\in \partial \Omega$ such that $d_{\epsilon}=|x_{\epsilon}-y_{\epsilon}|$. Let $\overline v_{\epsilon}=M_{\epsilon}^{-1}u_{\epsilon}(y_{\epsilon}+r_{\epsilon}x)$. By a reflection argument, similar to the case 1, we have $\overline v_{\epsilon}\rightarrow 1$ in $C^{1}(B_{R}^{+})$ for $||\overline v_{\epsilon}||_{L^{\infty}(\overline {B_{R}^{+}})}=1$. This contradicts $\overline v_{\epsilon}(0)=0$.

\end{proof}

Set
\begin{equation*}
\Omega_{\epsilon}=\{x\in \mathbb{R}^{n}|x_{\epsilon}+r_{\epsilon}x\in \Omega \}.
\end{equation*}
By Lemma \ref{4-33}, we have $\lim_{\epsilon\rightarrow 0}\frac{dist(x_{\epsilon},\partial \Omega)}{r_{\epsilon}}\rightarrow +\infty$, then $\Omega_{\epsilon}\rightarrow \mathbb{R}^{n}$. Let $w_{\epsilon}$ be define in (\ref{4-6}) and  $w$ be define in (\ref{4-13}). Similar arguments to Case 1 imply that $w_{\epsilon}\rightarrow w$ in $C_{loc}^{1}(\mathbb{R}^{n})$. We proceed as in Case 1, $M_{\epsilon}^{\frac{1}{n-1}}u_{\epsilon}\rightharpoonup \overline G_{\alpha}$ weakly in $W_{0}^{1,n}(\Omega)$ , and in $C^{1}(\Omega)$, where $\overline G$ satisfies the following equation:
\begin{equation*}
\left\{
\begin{array}{rrrl}
-Q_{n}\overline G_{\alpha} &=&0   \qquad & in~~~\Omega,\\
\overline G_{\alpha}&=& 0     \qquad & in~~~\partial\Omega.
\end{array}
\right.
\end{equation*}
The above equation has a unique solution $\overline G=0$. Hence
\begin{equation}\label{4-34}
M_{\epsilon}^{\frac{1}{n-1}}u_{\epsilon}\rightharpoonup 0~~~~~~~~ weakly~~in~~W_{0}^{1,n}(\Omega),~~~~~~M_{\epsilon}^{\frac{1}{n-1}}u_{\epsilon}\rightarrow 0~~~~~~~ in~~~C^{1}(\overline \Omega\backslash\{x_{0}\}).
\end{equation}
This is all we need to know about the convergence behavior of $u_{\epsilon}$ when the concentration point $x_{0}$ lies on the boundary of $\Omega$.

{\bf{The proof point (1) of Theorem \ref{1-04}}} . If $M_{\epsilon}$ is bounded, elliptic estimates implies that the Theorem holds. If $M_{\epsilon}\rightarrow +\infty$,
then we have $||u_{\epsilon}||_{L^{n}(\Omega)}\rightarrow 0$. A straightforward calculation gives
\begin{eqnarray*}
J_{\lambda_{n}-\epsilon}^{\alpha}(u_{\epsilon})&=&\int_{\Omega}
e^{(\lambda_{n}-\epsilon)|u_{\epsilon}|^{\frac{n}{n-1}}((1+\alpha||u_{\epsilon}||_{L^{n}(\Omega)}^{n})^{\frac{1}{n-1}}-1)}
e^{(\lambda_{n}-\epsilon)|u_{\epsilon}|^{\frac{n}{n-1}}}dx\\
&\leq& e^{\lambda_{n}M_{\epsilon}^{\frac{n}{n-1}}((1+\alpha||u_{\epsilon}||_{L^{n}(\Omega)}^{n})^{\frac{1}{n-1}}-1)}
\int_{\Omega}e^{\lambda_{n}|u_{\epsilon}|^{\frac{n}{n-1}}}dx\\
&=&e^{\frac{\lambda_{n}\alpha}{n-1}||M_{\epsilon}^{\frac{1}{n-1}}u_{\epsilon}||_{L^{n}(\Omega)}^{n}+M_{\epsilon}^{-\frac{n}{n-1}}O(||M_{\epsilon}^{\frac{1}{n-1}}u_{\epsilon}||_{L^{n}(\Omega)}^{2n})}
\int_{\Omega}e^{\lambda_{n}|u_{\epsilon}|^{\frac{n}{n-1}}}dx.
\end{eqnarray*}
Notice that $\alpha$ satisfies $0\leq \alpha<\lambda_{1}(\Omega)$. When $x_{0}\in \Omega$, we have $||M_{\epsilon}^{\frac{1}{n-1}}u_{\epsilon}||_{L^{n}(\Omega)}\rightarrow ||G_{\alpha}||_{L^{n}(\Omega)}^{n}$, when $x_{0}\in \partial \Omega$,$||M_{\epsilon}^{\frac{1}{n-1}}u_{\epsilon}||_{L^{n}(\Omega)}\rightarrow 0$.
Hence, together with Lemma \ref{3-10} and (\ref{1-02}) completes the proof of point (1) of Theorem \ref{1-04}.

\section{proof of theorem \ref{1-05}}
In this section, we will prove our main Theorem.  we first give a Lemma in \cite{ZZ}.

 \begin{lm}\label{5-1}
 Assume that ${u_{\epsilon}}$ is a normalized concentrating sequence in $W_{0}^{1, n}(\mathcal{W}_{\rho})$ with a blow up point at the origin, i.e.
 \begin{enumerate}
\item [(1)] $\int_{\mathcal{W}_{\rho}}F^{n}(\nabla u_{\epsilon})dx=1$,
\item [(2)] $ u_{\epsilon}\rightharpoonup 0 \text{ weakly in }  W_{0}^{1, n}(\mathcal{W}_{\rho})$,
\item [(3)] for any  $0<r<\rho$, $ \int_{\mathcal{W}_{\rho}\backslash \mathcal{W}_{r}}F^{n}(\nabla u_{\epsilon})dx\rightarrow 0$.
\end{enumerate}

 Then
 \begin{equation}\label{5-2}
  \limsup_{\epsilon\rightarrow 0}\int_{\mathcal{W}_{\rho}}(e^{\lambda_{n}|u_{\epsilon}|^{\frac{n}{n-1}}}-1)dx\leq\kappa_{n} \rho^{n}e^{1+\frac{1}{2}+\cdots+\frac{1}{n-1}}.
 \end{equation}
\end{lm}

Motivated by the arguments in \cite{CC, WY,Y1}, we first compute an upper bound of $T_0$ if $u_\epsilon$ blows up.

\begin{lm}\label{5-00}
If $\limsup_{\epsilon\rightarrow 0}||u_{\epsilon}||_{\infty}=\infty$,  then
\begin{equation}\label{5-01}
\sup_{u\in \mathcal{H}}J_{\lambda_{n}}^{\alpha}(u)\leq |\Omega|+\kappa_{n} e^{\lambda_{n}C_{G}+1+\frac{1}{2}+\cdots+\frac{1}{n-1}}.
\end{equation}
\end{lm}

\begin{proof}
\textbf{Case 1.} $x_{0}$ lies in the interior of $\Omega$.

Note that $x_0$ is a blow-up point of $u_\varepsilon $. We chose $\mathcal{W}_{\delta}(x_{0}) \subset \Omega$ for sufficient small $\delta>0$. By \eqref{4-31} we have
\begin{eqnarray}\label{5-02}
& &\int_{\Omega\backslash \mathcal{W}_{\delta}(x_{0})}F^{n}(\nabla G_{\alpha})dx\nonumber\\
&=&\int_{\partial (\Omega\backslash \mathcal{W}_{\delta}(x_{0}))}G_{\alpha}F^{n-1}(\nabla G_{\alpha})\langle F_{\xi}(\nabla G_{\alpha}),\nu\rangle ds-\int_{\Omega\backslash \mathcal{W}_{\delta}(x_{0})}G_{\alpha} div(F^{n-1}(\nabla G_{\alpha})F_{\xi}(\nabla G_{\alpha})) dx\nonumber\\
& &+\alpha\int_{\Omega\backslash \mathcal{W}_{\delta}(x_{0})}|G_{\alpha}|^{n}dx\nonumber\\
&=&-\int_{\partial \mathcal{W}_{\delta}(x_{0})}G_{\alpha}F^{n-1}(\nabla G_{\alpha})\langle F_{\xi}(\nabla G_{\alpha}),\nu\rangle ds+\alpha\int_{\Omega\backslash \mathcal{W}_{\delta}(x_{0})}|G_{\alpha}|^{n}dx.
\end{eqnarray}

Due to (\ref{4-39}) and Lemma \ref{2-01}, we have on $\partial \mathcal{W}_{\delta}(x_{0})$
\begin{eqnarray}\label{5-03}
F(\nabla G_{\alpha})&=&F(-\frac{1}{(n\kappa_{n})^{\frac{1}{n-1}}}\frac{\nabla F^{o}(x-x_{0})}{F^{o}(x-x_{0})}+o(\frac{1}{F^{o}(x-x_{0})}))\nonumber\\
&=&\frac{1}{(n\kappa_{n})^{\frac{1}{n-1}}\delta}+o(\frac{1}{\delta}).
\end{eqnarray}
and \begin{eqnarray}\label{5-04}
\langle F_{\xi}(\nabla G_{\alpha}),\nu\rangle&=&\langle F_{\xi}(\nabla G_{\alpha}),\frac{\nabla F^{o}(x-x_{0})}{|\nabla F^{o}(x-x_{0})|}\rangle\nonumber\\
&=&\langle F_{\xi}(\nabla G_{\alpha}),(-(n\kappa_{n})^{\frac{1}{n-1}}F^{o}(x-x_{0}))\frac{\nabla G_{\alpha}-o(\frac{1}{F^{o}(x-x_{0})})}{|\nabla F^{o}(x-x_{0})|}\rangle\nonumber\\
&=&-(n\kappa_{n})^{\frac{1}{n-1}}\delta(\frac{F(\nabla G_{\alpha})}{|\nabla F^{o}(x-x_{0})|}-\frac{o(\frac{1}{\delta})}{|\nabla F^{o}(x-x_{0})|})\nonumber\\
&=&-(1+o_{\delta}(1))\frac{1}{|\nabla F^{o}(x-x_{0})|}.
\end{eqnarray}
where $o_{\delta}(1)\rightarrow 0$ as $\delta\rightarrow 0$.
Putting (\ref{4-39}), (\ref{5-03}), (\ref{5-04}) into  (\ref{5-02}), we obtain
\begin{equation}\label{5-05}
\int_{\Omega\backslash \mathcal{W}_{\delta}(x_{0})}F^{n}(\nabla G_{\alpha})dx=-\frac{1}{(n\kappa_{n})^{\frac{1}{n-1}}}\log \delta+C_{G}+\alpha\int_{\Omega\backslash \mathcal{W}_{\delta}(x_{0})}|G_{\alpha}|^{n}dx+o_{\delta}(1)
\end{equation}

Hence from Lemma \ref{4-30} we have
\begin{equation}\label{5-06}
\int_{\Omega\backslash \mathcal{W}_{\delta}(x_{0})}F^{n}(\nabla u_{\epsilon})dx=\frac{1}{M_{\epsilon}^{\frac{n}{n-1}}}(-\frac{1}{(n\kappa_{n})^{\frac{1}{n-1}}}\log \delta+C_{G}+\alpha\int_{\Omega\backslash \mathcal{W}_{\delta}(x_{0})}|G_{\alpha}|^{n}dx+o_{\delta}(1)+o_{\epsilon}(1)),
\end{equation}
where $o_{\epsilon}(1)\rightarrow 0$ as $\epsilon\rightarrow 0$.

Next we let $b_{\epsilon}=\sup_{\partial \mathcal{W}_{\delta}(x_{0})}u_{\epsilon}$  and  $\overline u_{\epsilon}=(u_{\epsilon}-b_{\epsilon})^{+}$. Then $\overline u_{\epsilon}\in W_{0}^{1,n}(\mathcal{W}_{\delta}(x_{0}))$.  From (\ref{5-06}) and the fact that $\int_{\mathcal{W}_{\delta}(x_{0})}F^{n}(\nabla u_{\epsilon})dx=1- \int_{\Omega\backslash \mathcal{W}_{\delta}(x_{0})}F^{n}(\nabla u_{\epsilon})dx$,
we have
$$\int_{\mathcal{W}_{\delta}(x_{0})}F^{n}(\nabla\overline u_{\epsilon})dx= \tau_{\epsilon} \leq 1-\frac{1}{M_{\epsilon}^{\frac{n}{n-1}}}(-\frac{1}{(n\kappa_{n})^{\frac{1}{n-1}}}\log \delta+C_{G}+\alpha\int_{\Omega\backslash \mathcal{W}_{\delta}(x_{0})}|G_{\alpha}|^{n}dx+o_{\delta}(1)+o_{\epsilon}(1)).   $$
By Lemma \ref{5-1},
\begin{equation*}
\limsup_{\epsilon\rightarrow 0}\int_{\mathcal{W}_{\delta}(x_{0})}(e^{\lambda_{n}|\overline u_{\epsilon}/\tau_{\epsilon}^{1/n}|^{\frac{n}{n-1}}}-1)dx\leq \kappa_{n} \delta^{n}e^{1+\frac{1}{2}+\cdots+\frac{1}{n-1}}.
\end{equation*}

Now we focus on the estimate in the bubbling domain $\mathcal{W}_{Rr_{\epsilon}}(x_{\epsilon})$. According to the rescaling functions in Section 4, we can assume that  $w_{\epsilon}\rightarrow w$  in  $C_{loc}^{1}(\mathbb{R}^{n})$,  and whence $u_{\epsilon}=M_{\epsilon}+o_{\epsilon}(R)$, where  $o_{\epsilon}(R)\rightarrow 0$  as $\epsilon\rightarrow 0$ for any fixed $R>0$. Then from  Lemma \ref{4-30} we have
\begin{eqnarray}\label{5-07}
\alpha_{\epsilon}|u_{\epsilon}|^{\frac{n}{n-1}}&\leq& \lambda_{n}(1+\alpha||u_{\epsilon}||_{L^{n}(\Omega)}^{n})^{\frac{n}{n-1}}(\overline u_{\epsilon}+b_{\epsilon})^{\frac{n}{n-1}}\nonumber\\
&\leq&\lambda_{n}\overline u_{\epsilon}^{\frac{n}{n-1}}+\frac{\lambda_{n}\alpha}{n-1}||G_{\alpha}||_{L^{n}(\Omega)}^{n}+\frac{n}{n-1}\alpha_{\epsilon}b_{\epsilon}\overline u_{\epsilon}^{\frac{1}{n-1}}+o_{\epsilon}(1),
\end{eqnarray}
and
\begin{equation}\label{5-08}
b_{\epsilon}\overline u_{\epsilon}^{\frac{1}{n-1}}=-\frac{1}{(n\kappa_{n})^{\frac{1}{n-1}}}\log \delta+C_{G}+o_{\delta}(1)+o_{\epsilon}(1).
\end{equation}
Notice that
\begin{eqnarray}\label{5-09}
\lambda_{n}\overline u_{\epsilon}^{\frac{n}{n-1}}&\leq&\frac{ \lambda_{n}\overline u_{\epsilon}^{\frac{n}{n-1}}}{\tau_{\epsilon}^{\frac{1}{n-1}}}-\frac{\lambda_{n}}{n-1}(-\frac{1}{(n\kappa_{n})^{\frac{1}{n-1}}}\log \delta+\alpha||G_{\alpha}||_{L^{n}(\Omega)}^{n}+C_{G}+o_{\delta}(1)+o_{\epsilon}(1))\nonumber\\
&=&\frac{\lambda_{n}\overline u_{\epsilon}^{\frac{n}{n-1}}}{\tau_{\epsilon}^{\frac{1}{n-1}}}+\frac{n}{n-1}\log \delta-\frac{\alpha\lambda_{n}}{n-1}||G_{\alpha}||_{L^{n}(\Omega)}^{n}-\frac {\lambda_{n}}{n-1} C_{G}+o_{\delta}(1)+o_{\epsilon}(1).
\end{eqnarray}
Combining  (\ref{5-07})-(\ref{5-09}), we obtain in $\mathcal{W}_{Rr_{\epsilon}}(x_{\epsilon})$
\begin{equation*}
\alpha_{\epsilon}|u_{\epsilon}|^{\frac{n}{n-1}}\leq \frac{ \lambda_{n}\overline u_{\epsilon}^{\frac{n}{n-1}}}{\tau_{\epsilon}^{\frac{1}{n-1}}}-n\log \delta+\lambda_{n}C_{G}+o_{\delta}(1)+o_{\epsilon}(1).
\end{equation*}
Therefore, we have
\begin{eqnarray*}
& & \limsup_{\epsilon\rightarrow 0}\int_{\mathcal{W}_{Rr_{\epsilon}}(x_{\epsilon})}e^{\alpha_{\epsilon}|u_{\epsilon}|^{\frac{n}{n-1}}}dx\\
&\leq &  \delta^{-n} e^{\lambda_{n}C_{G}+o_{\delta}(1)}\limsup_{\epsilon\rightarrow 0}\int_{\mathcal{W}_{Rr_{\epsilon}}(x_{\epsilon})}(e^{\lambda_{n}|\overline u_{\epsilon}/\tau_{\epsilon}^{1/n}|^{\frac{n}{n-1}}}-1)dx\\
&\leq & \delta^{-n} e^{\lambda_{n}C_{G}+o_{\delta}(1)}\limsup_{\epsilon\rightarrow 0}\int_{\mathcal{W}_{\delta}(0)}(e^{\lambda_{n}|\overline u_{\epsilon}/\tau_{\epsilon}^{1/n}|^{\frac{n}{n-1}}}-1)dx\\
&\leq & \delta^{-n} e^{\lambda_{n}C_{G}+o_{\delta}(1)}\kappa_{n} \delta^{n}e^{1+\frac{1}{2}+\cdots+\frac{1}{n-1}}.
\end{eqnarray*}
Taking $\delta\rightarrow 0$, we have
\begin{equation*}
\limsup_{\epsilon\rightarrow 0}\int_{\mathcal{W}_{Rr_{\epsilon}}(x_{\epsilon})}e^{\alpha_{\epsilon}|u_{\epsilon}|^{\frac{n}{n-1}}}dx\leq \kappa_{n} e^{\lambda_{n}C_{G}+1+\frac{1}{2}+\cdots+\frac{1}{n-1}}.
\end{equation*}
Then by the Lemma \ref{4-21}, we obtain
\begin{equation*}
\limsup_{\epsilon\rightarrow 0}\int_{\Omega}e^{\alpha_{\epsilon}|u_{\epsilon}|^{\frac{n}{n-1}}}dx\leq |\Omega|+\kappa_{n} e^{\lambda_{n}C_{G}+1+\frac{1}{2}+\cdots+\frac{1}{n-1}}.
\end{equation*}
It follows Lemma \ref{3-10} to get
 $\sup_{u\in \mathcal{H}}J_{\lambda_{n}}^{\alpha}(u)\leq |\Omega|+\kappa_{n} e^{\lambda_{n}C_{G}+1+\frac{1}{2}+\cdots+\frac{1}{n-1}}.$

\
\

 \noindent\textbf{Case 2.} $x_{0}$ lies on the boundary of $\Omega$.

We proceed and use the same notions as in case 1.  By (\ref{4-34}), $M_{\epsilon}^{\frac{1}{n-1}}u_{\epsilon}\rightharpoonup 0$ weakly in $W_{0}^{1,q}(\Omega)$ for any $1<q<n$, and in $C^{1}(\overline{\Omega}\backslash\{ x_{0}\})$. Hence
\begin{equation}\label{5-010}
\int_{\mathcal{W}_{\delta}(x_{\epsilon})}F^{n}(\nabla \overline u_{\epsilon})dx\leq \tau_{\epsilon}=1-\frac{o_{\epsilon}(1)}{M_{\epsilon}^{n/(n-1)}},
\end{equation}
and we have in $\omega_{Rr_{\epsilon}}(x_{\epsilon})$,
\begin{equation}\label{5-011}
\alpha_{\epsilon}|u_{\epsilon}|^{\frac{n}{n-1}}\leq \lambda_{n}|\overline u_{\epsilon}/\tau_{\epsilon}^{1/n}|^{\frac{n}{n-1}}+o_{\epsilon}(1).
\end{equation}
Combining (\ref{5-010}), (\ref{5-011}) and Lemma \ref{4-21}, we have
\begin{equation}\label{5-012}
\limsup_{\epsilon\rightarrow 0}\int_{\Omega}e^{\alpha_{\epsilon}|u_{\epsilon}|^{\frac{n}{n-1}}}dx\leq |\Omega|+O(\delta^{n})e^{1+\frac{1}{2}+\cdots+\frac{1}{n-1}}.
\end{equation}
Letting $\delta\rightarrow 0$, (\ref{5-012}) together with Lemma \ref{3-10} gives $\sup_{u\in \mathcal{H}}J_{\lambda_{n}}^{\alpha}(u)\leq |\Omega|$, which is impossible. Therefore we conclude that $x_{0}$ cannot lie on $\partial\Omega$.

\end{proof}

\
\

In Lemma \ref{5-00} we have got the upper bound of $\sup_{u\in \mathcal{H}}J_{\lambda_{n}}^{\alpha}(u)$ if $u_\epsilon$ blows up.  Next we will construct an explicit test function to get the lower bound of $\sup_{u\in \mathcal{H}}J_{\lambda_{n}}^{\alpha}(u)$, which  will contradict the the upper bound of $\sup_{u\in \mathcal{H}}J_{\lambda_{n}}^{\alpha}(u)$. Thus we get a contradiction and consequently we complete the proof of Theorem. The similar arguments can be seen in \cite{YZ1,Z}.
\begin{lm}
There holds
\begin{equation}\label{5-013}
\sup_{u\in \mathcal{H}}J_{\lambda_{n}}^{\alpha}(u)> |\Omega|+\kappa_{n} e^{\lambda_{n}C_{G}+1+\frac{1}{2}+\cdots+\frac{1}{n-1}}.
\end{equation}
\end{lm}
\begin{proof}
Define a sequence of functions in $\Omega$ by
\begin{equation*}\label{5-014}
\phi_{\epsilon}=
\left\{
\begin{array}{ll}
\frac{C+C^{-\frac{1}{n-1}}(-\frac{n-1}{\lambda_{n}}\log(1+\kappa_{n}^{\frac{1}{n-1}}(\frac{F^{o}(x-x_{0})}
{\epsilon})^{\frac{n}{n-1}})+b)}{(1+\alpha C^{\frac{-n}{n-1}}||G_{\alpha}||_{L^{n}(\Omega)}^{n})^{\frac{1}{n}}},& x\in \overline {\mathcal{W}_{R\epsilon}(x_{0})},\\
\frac{C^{-\frac{1}{n-1}}(G-\eta \psi)}{(1+\alpha C^{\frac{-n}{n-1}}||G_{\alpha}||_{L^{n}(\Omega)}^{n})^{\frac{1}{n}}},& x\in \mathcal{W}_{2R\epsilon}(x_{0})\backslash\overline {\mathcal{W}_{R\epsilon}(x_{0})},\\
\frac{C^{-\frac{1}{n-1}}G}{(1+\alpha C^{\frac{-n}{n-1}}||G_{\alpha}||_{L^{n}(\Omega)}^{n})^{\frac{1}{n}}},& x\in \Omega\backslash \mathcal{W}_{2R\epsilon}(x_{0}),
\end{array}
\right.
\end{equation*}
where $G$ and $\psi$ are functions given in (\ref{4-39}), $R=-\log \epsilon,~~~\eta\in C_{0}^{1}(\mathcal{W}_{2R\epsilon}(x_{0}))$~~~satisfying that $\eta=1$ on $ \mathcal{W}_{R\epsilon}(x_{0})$  and  $|\nabla \eta|\leq \frac{2}{R\epsilon}$, $b$  and  $C$  are constants depending only on $\epsilon$  to be determined later. Clearly $\mathcal{W}_{2R\epsilon}(x_{0})\subset \Omega$ provided that $\epsilon$ is sufficiently small. In order to assure that $\phi_{\epsilon}\in W_{0}^{1,n}(\Omega)$, we set
$$C+C^{-\frac{1}{n-1}}(-\frac{n-1}{\lambda_{n}}\log(1+\kappa_{n}^{\frac{1}{n-1}}R^{\frac{n}{n-1}})+b)=C^{-\frac{1}{n-1}}(-\frac{1}{(n\kappa_{n})^{\frac{1}{n-1}}}\log (R\epsilon)+C_{G}),    $$
which gives
\begin{equation}\label{5-015}
C^{\frac{n}{n-1}}=-\frac{1}{(n\kappa_{n})^{\frac{1}{n-1}}}\log \epsilon+\frac{1}{\lambda_{n}}\log \kappa_{n}-b+C_{G}+O(R^{-\frac{n}{n-1}}).
\end{equation}
Next we make sure that $\int_{\Omega}F^{n}(\nabla \phi_{\epsilon})dx=1$.

By the coarea formula (\ref{2-03}), we have
\begin{eqnarray*}
\int_{\mathcal{W}_{R\epsilon}(x_{0})}\frac{(\frac{F^{o}(x-x_{0})}{\epsilon})^{\frac{n}{n-1}}\frac{1}{\epsilon^{n}}}{(1+\kappa_{n}^{\frac{1}{n-1}}(\frac{F^{o}(x-x_{0})}{\epsilon})^{\frac{n}{n-1}})^{n}}dx
&=&n\kappa_{n}\int_{0}^{R\epsilon}\frac{(\frac{s}{\epsilon})^{\frac{n}{n-1}}\frac{1}{\epsilon^{n}}}{(1+\kappa_{n}^{\frac{1}{n-1}}(\frac{s}{\epsilon})^{\frac{n}{n-1}})^{n}}s^{n-1}ds \\
&=&\frac{n-1}{\kappa_{n}^{\frac{1}{n-1}}}\int_{0}^{\kappa_{n}^{\frac{1}{n-1}}
R^{\frac{n}{n-1}}}\frac{t^{n-1}}{(1+t)^{n}}dt.
\end{eqnarray*}
Then it follows that
\begin{eqnarray}\label{5-016}
\int_{\mathcal{W}_{R\epsilon}(x_{0})}F^{n}(\nabla \phi_{\epsilon})dx&=&\frac{n-1}{\lambda_{n}(C^{\frac{n}{n-1}}+\alpha||G_{\alpha}||_{L^{n}(\Omega)}^{n})}\int_{0}^{\kappa_{n}^{\frac{1}{n-1}}R^{\frac{n}{n-1}}}
\frac{t^{n-1}}{(1+t)^{n}}dt\nonumber\\
&=&\frac{n-1}{\lambda_{n}(C^{\frac{n}{n-1}}+\alpha||G_{\alpha}||_{L^{n}(\Omega)}^{n})}\int_{0}^{\kappa_{n}^{\frac{1}{n-1}}R^{\frac{n}{n-1}}}
\frac{(t+1-1)^{n-1}}{(1+t)^{n}}dt\nonumber\\
&=&\frac{n-1}{\lambda_{n}(C^{\frac{n}{n-1}}+\alpha||G_{\alpha}||_{L^{n}(\Omega)}^{n})}(\sum_{k=0}^{n-2}\frac{C_{n-1}^{k}(-1)^{n-1-k}}{n-k-1}\nonumber\\
& &+\log(1+\kappa_{n}^{\frac{1}{n-1}}R^{\frac{n}{n-1}})+O(R^{-\frac{n}{n-1}}))\nonumber\\
&=&\frac{n-1}{\lambda_{n}(C^{\frac{n}{n-1}}+\alpha||G_{\alpha}||_{L^{n}(\Omega)}^{n})}(-(1+\frac{1}{2}+\cdots+\frac{1}{n-1})\nonumber\\
& &+\log(1+\kappa_{n}^{\frac{1}{n-1}}R^{\frac{n}{n-1}})+O(R^{-\frac{n}{n-1}})),
\end{eqnarray}
where we have used the fact that
\begin{equation*}
-\sum_{k=0}^{n-2}\frac{C_{n-1}^{k}(-1)^{n-1-k}}{n-k-1}=1+\frac{1}{2}+\cdots+\frac{1}{n-1}.
\end{equation*}
Noting that $\psi(x)$ satisfies that $|\nabla\psi(x)|=o(\frac{1}{F^0(x-x_{0})})$ as $x\rightarrow x_{0} $, and using Lemma \ref{2-01} and (\ref{4-39}), we have
$$~\int_{\mathcal{W}_{2R\epsilon}(x_{0})\backslash \mathcal{W}_{R\epsilon}(x_{0}) }F^{n}(\nabla G_{\alpha})-F^{n}(\nabla (G_{\alpha}-\eta\psi))dx=o_{R\epsilon}(1).$$
Then together with (\ref{5-05}), we have
\begin{eqnarray}\label{5-017}
 & & \int_{\Omega\backslash \mathcal{W}_{R\epsilon}(x_{0})}F^{n}(\nabla \phi_{\epsilon})dx\nonumber \\
&=& \frac{1}{C^{\frac{n}{n-1}}+\alpha||G_{\alpha}||_{L^{n}(\Omega)}^{n}}(\int_{\Omega\backslash \mathcal{W}_{R\epsilon}(x_{0})}F^{n}(\nabla G_{\alpha})dx\nonumber\\
& &-\int_{ \mathcal{W}_{2R\epsilon}(x_{0})\backslash \mathcal{W}_{R\epsilon}(x_{0})}F^{n}(\nabla G_{\alpha})
-F^{n}(\nabla (G_{\alpha}-\eta\psi))dx)\nonumber\\
 &=&\frac{1}{C^{\frac{n}{n-1}}+\alpha||G_{\alpha}||_{L^{n}(\Omega)}^{n}}(-\frac{1}{(n\kappa_{n})^{\frac{1}{n-1}}}\log (R\epsilon)\nonumber\\
 & &+\frac{\alpha\lambda_{n}}{n-1}||G_{\alpha}||_{L^{n}(\Omega)}^{n}+C_{G}+o_{R\epsilon}(1)).
\end{eqnarray}

Putting (\ref{5-016}), (\ref{5-017}) together, we have
\begin{eqnarray*}
\int_{\Omega}F^{n}(\nabla \phi_{\epsilon})dx&=&\frac{n-1}{\lambda_{n}(C^{\frac{n}{n-1}}+\alpha||G_{\alpha}||_{L^{n}(\Omega)}^{n})}(-\frac{n}{n-1}\log\epsilon+\frac{1}{n-1}\log\kappa_{n}+\frac{\lambda_{n}}{n-1}C_{G}\\
& &+\frac{\alpha\lambda_{n}}{n-1}||G_{\alpha}||_{L^{n}(\Omega)}^{n}-(1+\frac{1}{2}+\cdots+\frac{1}{n-1})+O(R^{-\frac{n}{n-1}})+o_{R\epsilon}(1)).
\end{eqnarray*}
Since $\int_{\Omega}F^{n}(\nabla \phi_{\epsilon})dx=1$, we have
\begin{eqnarray}\label{5-018}
C^{\frac{n}{n-1}}&=&\frac{n-1}{\lambda_{n}}(-\frac{n}{n-1}\log\epsilon+\frac{1}{n-1}\log\kappa_{n}+\frac{\lambda_{n}}{n-1}C_{G}\nonumber\\
& &-(1+\frac{1}{2}+\cdots+\frac{1}{n-1})+O(R^{-\frac{n}{n-1}})+o_{R\epsilon}(1).
\end{eqnarray}
Consequently from (\ref{5-015}), we have
\begin{equation}\label{5-019}
b=\frac{(n-1)}{\lambda_n}(1+\frac{1}{2}+\cdots+\frac{1}{n-1})+O(R^{-\frac{n}{n-1}})+o_{R\epsilon}(1).
\end{equation}
Since
\begin{equation*}
||\phi_{\epsilon}||_{L^{n}(\Omega)}^{n}=\frac{||G_{\alpha}||_{L^{n}(\Omega)}^{n}+O(C^{\frac{n^{2}}{n-1}}R^{n}\epsilon^{n})+O((R\varepsilon)^{n}(\log(R\epsilon)^{n}))}
{C^{\frac{n}{n-1}}+\alpha||G_{\alpha}||_{L^{n}(\Omega)}^{n}},
\end{equation*}
using the inequality
\begin{equation*}
(1+t)^{\frac{1}{n+1}}\geq 1-\frac{t}{n-1},~~~\text {for}~~~t~~\text{small }.
\end{equation*}

In view of (\ref{5-018}) and (\ref{5-019}), there holds in $\mathcal{W}_{R\epsilon}(x_{0})$,
\begin{eqnarray*}
& & \lambda_{n}|\phi_{\epsilon}(x)|^{\frac{n}{n-1}}(1+\alpha||\phi_{\epsilon}||_{L^{n}(\Omega)}^{n})^{\frac{1}{n-1}}\\
&\geq& \lambda_{n}C^{\frac{n}{n-1}}-n\log(1+\kappa_{n}^{\frac{1}{n-1}}(\frac{F^{o}(x-x_{0})}
{\epsilon})^{\frac{n}{n-1}})+\frac{n\lambda_{n}}{n-1}b-\frac{\alpha^{2}\lambda_{n}||G_{\alpha}||_{L^{n}(\Omega)}^{2n}}{(n-1)C^{\frac{n}{n-1}}}\\
& &+O(C^{-\frac{2n}{n-1}})+O(C^{\frac{n^{2}}{n-1}}(R\epsilon)^{n})+O((R\epsilon)^{n}(-\log(R\epsilon))^{n})\\
&\geq&-n\log\epsilon+\log\kappa_{n}+\lambda_{n}C_{G}+(1+\frac{1}{2}+\cdots+\frac{1}{n-1})\\
& &-n\log(1+\kappa_{n}^{\frac{1}{n-1}}(\frac{F^{o}(x-x_{0})}
{\epsilon})^{\frac{n}{n-1}})-\frac{\alpha^{2}\lambda_{n}||G_{\alpha}||_{L^{n}(\Omega)}^{2n}}{(n-1)C^{\frac{n}{n-1}}}\\
& &+O(C^{-\frac{2n}{n-1}})+O(R^{\frac{-n}{n-1}})+o_{R\epsilon}(1).
\end{eqnarray*}
where we have used the inequality $|1+t|^{\frac{n}{n-1}}\geq 1+\frac{n}{n-1}t+O(t^{3})$ for small $t$.
By using the fact
\begin{equation*}
\sum_{k=0}^{n-2}\frac{C_{n-2}^{k}(-1)^{n-k-2}}{n-k-1}=\frac{1}{n-1},
\end{equation*}
one can get
\begin{eqnarray*}
& &\int_{\mathcal{W}_{R\epsilon}(x_{0})}e^{-n\log\epsilon-n\log(1+\kappa_{n}^{\frac{1}{n-1}}(\frac{F^{o}(x-x_{0})}
{\epsilon})^{\frac{n}{n-1}})}dx\\
&=&\frac{1}{\epsilon^{n}}\int_{\mathcal{W}_{R\epsilon}(x_{0})}\frac{1}{(1+\kappa_{n}^{\frac{1}{n-1}}(\frac{F^{o}(x-x_{0})}
{\epsilon})^{\frac{n}{n-1}})^{n}}dx\\
&=&(n-1)\int_{0}^{\kappa_{n}^{\frac{1}{n-1}}R^{\frac{n}{n-1}}}\frac{t^{n-2}}{(1+t)^{n}}dt\\
&=&(n-1)\int_{0}^{\kappa_{n}^{\frac{1}{n-1}}R^{\frac{n}{n-1}}}\frac{(t+1-1)^{n-2}}{(1+t)^{n}}dt\\
&\geq&(n-1)(\frac{1}{n-1}+O(R^{-\frac{n}{n-1}}))=1+O(R^{-\frac{n}{n-1}})).
\end{eqnarray*}
Then we obtain
\begin{eqnarray*}
\int_{\mathcal{W}_{R\epsilon}(x_{0})}e^{\lambda_{n}|\phi_{\epsilon}(x)|^{\frac{n}{n-1}}
(1+\alpha||\phi_{\epsilon}(x)||_{L^{n}(\Omega)}^{n})^{\frac{1}{n-1}}}dx&\geq& \kappa_{n}e^{\lambda_{n}C_{G}+(1+\frac{1}{2}+\cdots+\frac{1}{n-1})}(1-
\frac{\alpha^{2}\lambda_{n}||G_{\alpha}||_{L^{n}(\Omega)}^{2n}}{(n-1)C^{\frac{n}{n-1}}})\\
& &+O(C^{-\frac{2n}{n-1}})+O(R^{-\frac{n}{n-1}})+o_{R\epsilon}(1).
\end{eqnarray*}
On the other hand, since
\begin{equation*}
\int_{\mathcal{W}_{2R\epsilon}(x_{0})}|G_{\alpha}|^{\frac{n}{n-1}}dx=O((R\epsilon)^{n}\log^{\frac{n}{n-1}}(R\epsilon))=o_{R\epsilon}(1),
\end{equation*}
we obtain
\begin{eqnarray*}
& &\int_{\Omega\backslash \mathcal{W}_{R\epsilon}(x_{0})}e^{\lambda_{n}|\phi_{\epsilon}(x)|^{\frac{n}{n-1}}(1+\alpha||\phi_{\epsilon}(x)||_{L^{n}(\Omega)}^{n})^{\frac{1}{n-1}}}dx\\
&\geq & \int_{\Omega\backslash \mathcal{W}_{2R\epsilon}(x_{0})}(1+\lambda_{n}|\phi_{\epsilon}(x)|^{\frac{n}{n-1}})dx\\
&\geq& |\Omega|+\frac{\lambda_{n}||G_{\alpha}||_{L^{\frac{n}{n-1}}}^{\frac{n}{n-1}}}{C^{\frac{n}{(n-1)^{2}}}}+O(C^{\frac{-2n}{(n-1)^{2}}})+o_{R\epsilon}(1)
\end{eqnarray*}
Together with the above integral estimates on $\mathcal{W}_{R\epsilon}$, and the fact that
\begin{equation*}
C^{\frac{-2n}{(n-1)^{2}}}\rightarrow 0~~~\text{and}~~~R^{\frac{-n}{(n-1)}}\rightarrow 0~~~\text{as}~~~\epsilon\rightarrow 0.
\end{equation*}
Then we have
\begin{equation*}
\int_{\Omega}e^{\lambda_{n}|\phi_{\epsilon}(x)|^{\frac{n}{n-1}}(1+\alpha||\phi_{\epsilon}(x)||_{L^{n}(\Omega)}^{n})^{\frac{1}{n-1}}}dx>|\Omega|+\kappa_{n}e^{\lambda_{n}C_{G}+(1+\frac{1}{2}+\cdots+\frac{1}{n-1})},
\end{equation*}
provided that $\epsilon>0$  is chosen sufficiently small. Thus we get the conclusion of Lemma.
\end{proof}

\section{Asymptotic representation of $G_{\alpha}$}
In this section we will give the asymptotic representation of Green function $G_{\alpha}$, similarly to \cite{KV,WX1,Y1}

{\bf{The proof of Lemma \ref{4-66}}:}
Since $c_{k}^{\frac{n}{n-1}}u_{k}\geq 0$ in $\Omega\backslash \{0\}$, we have $G_{\alpha}\geq 0$ in $\Omega\backslash\{0\}$. Theorem 1 in \cite{S3} gives
\begin{equation} \label{7-20}
\frac{1}{K}\leq \frac{G_{\alpha}}{-\log r}\leq K \quad\text{in}\quad\Omega\backslash \{0\}
\end{equation}
for some constant $k>0$.
Assume $\Gamma(r)=-c(n)\log r$, $c(n)=(n\kappa_{n})^{-\frac{1}{n-1}}$. Let $G_{k}=\frac{G_{\alpha}(r_{k}x)}{\Gamma(r_{k})}$, which is defined in $\{x\in \mathbb{R}^{n}\backslash \{0\},\quad r_{k}x\in \mathcal{W}_{\delta} \}$ for some small $\delta>0$.Here $r_{k}\rightarrow 0$ as $k\rightarrow +\infty$.
then $G_{k}$ satisfy the equation
$$-\sum_{i=1}^{n}\frac{\partial}{\partial x_{i}}(F^{n-1}(\nabla G_{k})F_{\xi}(\nabla G_{k}))=\alpha r_{k}^{n}G_{k}^{n-1}.    $$
By theorem 1 in \cite{T2}, when $r_{k}\rightarrow 0$, $G_{k}$ converges to $G^{\ast}$ in $C_{loc}^{1}(\mathbb{R}^{n}\backslash\{0\})$ and $G^{\ast}$ is bounded, where $G^{\ast}$ satisfying
$$-\sum_{i=1}^{n}\frac{\partial}{\partial x_{i}}(F^{n-1}(\nabla G^{\ast})F_{\xi}(\nabla  G^{\ast}))=0  . $$
From serrin's result (see \cite{S1})and (\ref{7-20}), $0$ is a removable singularity and $G^{\ast}$ can be extended to $\hat{G}\in C^{1}(\mathbb{R}^{n})$. Consequently, form Liouville type theorem (see\cite{HKM}), $\hat{G}$ must be a constant. Let $\gamma_{k}=\sup_{\mathcal{W}_{\delta}\backslash\mathcal{W}_{r_{k}}}\frac{G_{\alpha}(x)}{\Gamma(x)}$, and
$\gamma=\lim_{k\rightarrow +\infty}\gamma_{k}, (\gamma>0).$ This means the constant function $\hat{G}=\gamma$.

Set
\begin{eqnarray}
G_{\eta}^{+}(x)=(\gamma+\eta)(\Gamma(x)-\Gamma(\delta))-c(n)(\gamma+\eta)(F^{o}(x)-\delta)+\sup_{\partial \mathcal{W}_{\delta}}G_{\alpha},\\
G_{\eta}^{-}(x)=(\gamma-\eta)(\Gamma(x)-\Gamma(\delta))-c(n)(\gamma-\eta)(F^{o}(x)-\delta)+\inf_{\partial \mathcal{W}_{\delta}}G_{\alpha}.
\end{eqnarray}
A straightforward calculation shows
\begin{eqnarray}
-Q_{n}G_{\eta}^{+}(x)=c^{n-1}(n)(\gamma+\eta)^{n-1}\frac{n-1}{F^{o}(x)}(\frac{1}{F^{o}(x)}+1)^{n-2},\\
-Q_{n}G_{\eta}^{-}(x)=c^{n-1}(n)(\gamma-\eta)^{n-1}\frac{n-1}{F^{o}(x)}(\frac{1}{F^{o}(x)}-1)^{n-2}.
\end{eqnarray}
By, for any fixed $0<\eta<\gamma$, we have
\begin{eqnarray*}
-Q_{n}G_{\eta}^{+}(x)\geq-Q_{n}G\qquad in \quad \mathcal{W}_{\delta}\backslash \mathcal{W}_{r_{k}},\\
G_{\eta}^{+}|_{\partial \mathcal{W}_{\delta}}\geq G_{\alpha}|_{\partial \mathcal{W}_{\delta}},\qquad G_{\eta}^{+}|_{\partial \mathcal{W}_{r_{k}}}\geq G_{\alpha}|_{\partial \mathcal{W}_{r_{k}}},
\end{eqnarray*}
provided that $\delta$ are sufficiently small and $r_{k}<\delta$. By the comparison principle (see\cite{XG}), we have
\begin{equation}
G_{\alpha}\leq (\gamma+\eta)\Gamma(x)+C_{\delta}\qquad in\quad \mathcal{W}_{\delta}\backslash \mathcal{W}_{r_{k}}
\end{equation}
for some constant $C_{\delta}$. Letting $\eta\rightarrow 0$ first, then $k\rightarrow\infty$, one has
\begin{equation*}
G_{\alpha}\leq \gamma\Gamma(x)+C_{\delta}\qquad in\quad \mathcal{W}_{\delta}\backslash \{0\}.
\end{equation*}
A similar argument gives $G_{\alpha}\geq \gamma \Gamma(x)+C_{\delta}^{'}$ in $\mathcal{W}_{\delta}\backslash \{0\}$ for some constant $C_{\delta}^{'}$. Hence $G_{\alpha}-\gamma\Gamma(x)$ is bounded in $L^{\infty}(\mathcal{W}_{\delta})$.

Next we prove the continuity of $G_{\alpha}-\gamma\Gamma(x)$ at $0$ .We look at the points where the bounded function $G_{\alpha}-\gamma\Gamma(x)$ achieves its supremum in
$\overline{\mathcal{W}_{\delta}}$.Set $\lambda=\sup_{\overline{\mathcal{W}_{\delta}}}(G_{\alpha}-\gamma\Gamma(x))$.

$\lambda$ achieves at some point in $\mathcal{W}_{\delta}\backslash \{0\}$, then $G_{\alpha}-\gamma\Gamma(x)-\gamma c(n)F^{o}(x)$ also achieves at some point in $\mathcal{W}_{\delta}\backslash \{0\}$. It follows from comparison principle (see\cite{D1}) that $G_{\alpha}-\gamma\Gamma(x)-\gamma c(n)F^{o}(x)$ is a constant, hence we have done.

$\lambda$ achieves at $0$, Set
\begin{equation*}
w_{r}(x)=G_{\alpha}(rx)-\gamma\Gamma (r)\qquad in\quad \mathcal{W}_{\frac{\delta}{r}} \backslash \{0\}.
\end{equation*}
The function $w_{r}$ satisfies $-Q_{n}(w_{r}(x))-\alpha r^{n}G_{\alpha}^{n-1}(rx)=0$. We also have $r^{n}G_{\alpha}^{n-1}(rx)\in L^{\infty}(\mathcal{W}_{\delta})$ and
$|w_{r}-\gamma\Gamma(x)|\leq C_{0}$ for $C_{0}=\sup_{\mathcal{W_{\delta}}\backslash\{0\}}|G_{\alpha}-\gamma\Gamma(x)|$. By Theorem 1 in \cite{T2}, when $r\rightarrow 0$,
$w_{r}\rightarrow w$ in $C_{loc}^{1}(\mathbb{R}^{n}\backslash\{0\})$, where $w\in C^{1}(\mathbb{R}^{n}\backslash\{0\})$ satisfies $-Q_{n}(w)=0$. For the sequence $\xi_{j}=\frac{x_{r_{j}}}{r_{j}},F^{o}(\xi_{j})=1$, which maybe assumed to converge to $\xi^{0}\in \partial \mathcal{W}_{1}$, we have
\begin{equation*}
w_{r_{j}}(\xi_{j})-\gamma\Gamma(\xi_{j})=G_{\alpha}(x_{r_{j}})-\gamma\Gamma(x_{r_{j}})\rightarrow \lambda.
\end{equation*}
Hence
\begin{equation*}
w(x)\leq \gamma\Gamma(x)+\lambda \qquad and \qquad w(\xi^{0})=\gamma\Gamma(\xi^{0})+\lambda.
\end{equation*}
By comparison principle (see \cite{XG}),  $w(x)=\gamma\Gamma(x)+\lambda$ and hence $w_{r}\rightarrow \gamma\Gamma(x)+\lambda$ in $C_{loc}^{1}(\mathbb{R}^{n}\backslash\{0\})$.
This implies
\begin{equation}\label{7-04}
\lim_{r\rightarrow 0}(G_{\alpha}(rx)-\gamma\Gamma(rx))=\lambda,\qquad\qquad \lim_{r\rightarrow 0}\nabla_{x} (G_{\alpha}(rx)-\Gamma(rx))=0.
\end{equation}
The above equalities lead to the continuity of $G_{\alpha}-\gamma\Gamma$ and $\lim_{x\rightarrow 0}F^{0}(x)\nabla(G_{\alpha}-\gamma\Gamma)=0$.

We assume $\sup_{x\in \mathcal{W}_{\delta}}(G_{\alpha}-\gamma\Gamma)=\sup_{F^{0}(x)=\delta}(G_{\alpha}-\gamma\Gamma)$, we define $w_{r}$ as the above, then $w_{r}\rightarrow w$ in $C_{loc}^{1}(\mathbb{R}^{n}\backslash \{0\})$ and $|w-\gamma\Gamma|\leq C_{0}.$ We now look at the points where $w-\gamma\Gamma$ achieves its supremum in $\mathbb{R}^{n}$.
Set $\tilde{\lambda}=\sup_{\mathbb{R}^{n}}(w-\gamma\Gamma).$

If $\tilde{\lambda}$ is achieved at some point in $\mathbb{R}^{n}\backslash \{0\}$, then $w-\gamma\Gamma$ equals to some constant by strong maximum principle\cite{G}  , which implies $G_{\alpha}(rx)-\gamma\Gamma(rx)\rightarrow \tilde{\lambda}$ in $C_{loc}^{1}(\mathbb{R}^{n}\backslash \{0\})$ as $r\rightarrow 0$. For any fixed $\epsilon>0$, there exists $n_{0}$ such that $n\geq n_{0}$ and $x\in \partial\mathcal{W}_{1}$, we have
\begin{equation*}
\gamma\Gamma(r_{n}x)+\tilde{\lambda}-\epsilon\leq G_{\alpha}(r_{n}x)\leq \gamma\Gamma(r_{n}x)+\tilde{\lambda}+\epsilon.
\end{equation*}
Applying maximum principle\cite{G} in $\mathcal{W}_{r_{n_{0}}}\backslash\mathcal{W}_{r_{n}}$ we obtain
\begin{equation*}
\gamma\Gamma(x)+\tilde{\lambda}-\epsilon\leq G_{\alpha}(x)\leq \gamma\Gamma(x)+\tilde{\lambda}+\epsilon,
\end{equation*}
which leads to    with $\lambda$ replaced by $\tilde{\lambda}$.

If $\tilde{\lambda}$ is achieved at $0$, we simply argue as in the above to deduce
\begin{equation}\label{7-05}
\lim_{x\rightarrow 0}(w-\gamma\Gamma)=\tilde{\lambda}\qquad and \quad hence \qquad \lim_{x\rightarrow 0}\lim_{r_{n}\rightarrow 0}(G_{\alpha}(r_{n}x)-\gamma\Gamma(r_{n}x))=\tilde{\lambda}.
\end{equation}

If $\tilde{\lambda}$ is achieved at $\infty$, the same idea as in case can be applied when we defined $\lambda(R)=\max_{\delta\leq F^{o}(x)\leq R}(w-\gamma\Gamma)=\max_{\partial \mathcal{W}_{R}}(w-\gamma\Gamma)$ and let $R$ tend to $\infty$. We obtain
\begin{equation}\label{7-06}
\lim_{x\rightarrow \infty}(w-\gamma\Gamma)=\tilde{\lambda},\qquad \lim_{x\rightarrow \infty}\lim_{r_{n}\rightarrow 0}(G_{\alpha}(r_{n}x)-\gamma\Gamma(r_{n}x))=\tilde{\lambda}.
\end{equation}
As long as we have (\ref{7-05})and (\ref{7-06}) , we can have use maximum principle\cite{G} again to conclude (\ref{7-04})  as before.

Integrating by parts on both sides of Eq.(\ref{4-31}) over $\mathcal{W}_{\delta}$, we have
\begin{equation}\label{7-16}
-\int_{\mathcal{W}_{\delta}}div(F^{n-1}(\nabla G_{\alpha})F_{\xi}(\nabla G_{\alpha}))dx+=1+\alpha\int_{\mathcal{W}_{\delta}}G_{\alpha}^{n-1}dx.
\end{equation}
Because $G_{\alpha}(x)=\gamma\Gamma(x)+o(1)$ and $\nabla G_{\alpha}(x)=\gamma\nabla\Gamma(x)+o(\frac{1}{F^{o}(x)}) $ as $x\rightarrow 0.$ Inserting the above two equalities into (\ref{7-16}) , then letting $\delta\rightarrow 0$, we obtain $\gamma=1$.

\end{document}